\documentclass{article}
\usepackage{fullpage}
\usepackage{multicol, float}
\usepackage{lineno,hyperref}
\usepackage{verbatim, euscript}
\usepackage{subfigure}
\usepackage{bbm,amssymb}
\usepackage{amsfonts}
\usepackage{amsmath, mathtools, xcolor}
\usepackage{pictexwd}
\usepackage{amsthm}
\usepackage{graphicx}
\usepackage{tikz}
\usetikzlibrary{positioning,arrows}
\usepackage{rotating}
\usepackage{color}
\usepackage{mathrsfs}

\usepackage{mathtools}
\mathtoolsset{showonlyrefs}
\newcommand{\eps}{\varepsilon}

\theoremstyle{plain}
\newtheorem{thm}{Theorem}

\newtheorem{lem}{Lemma}[section]
  
\newtheorem{cor}[lem]{Corollary}
\newtheorem{defi}[lem]{Definition}
\newtheorem{rem}[lem]{Remark}

\newtheorem{prop}[lem]{Proposition}

\definecolor{darkgreen}      {cmyk}{0.90,0.00,0.90,0.10}

\newcommand{\blue}[1]{\textcolor{black}{#1}}

\newcommand{\E}{\mathbb{E}}
\newcommand{\PP}{\mathbb{P}}
\newcommand{\1}{\mathbbm{1}} 
\newcommand{\N}{\mathbb{N}}

\newcommand{\R}{\mathbb{R}}

\newcommand{\EE}[1]{\mathbb{E} \left[ #1 \right]}

\begin{document}
	\title{On the fixation probability of an advantageous allele in a population with skewed offspring distribution}
	\author{Matthias Birkner\thanks{Johannes Gutenberg-Universität Mainz, birkner@mathematik.uni-mainz.de}, Florin Boenkost\thanks{%Faculty of Mathematics,
			University of Vienna, % Oskar-Morgenstern-Platz 1,1090, Vienna, Austria, \\
			florin.boenkost@univie.ac.at}, \\
		Iulia Dahmer\thanks{Goethe University Frankfurt, dahmer@med.uni-frankfurt.de}, Cornelia Pokalyuk\thanks{{\color{black} Universität zu Lübeck,  cornelia.pokalyuk@uni-luebeck.de}} }
	\date{\today}
	\maketitle
	
	\begin{abstract}
	Consider an advantageous allele that arises in a haploid population of size $N$ evolving in continuous time according to a skewed reproduction mechanism, which generates under neutrality genealogies lying in the domain of attraction of a Beta$(2-\alpha, \alpha)$-coalescent for $\alpha \in (1,2)$. We prove in a setting of moderate selection that the fixation probability $\pi_N$ of the advantageous allele is asymptotically equal to  $\alpha^{1/(\alpha-1)} s_N^{1/(\alpha-1)} $ , where $s_N$ is the selection strength of the advantageous allele. Our proof uses duality with a suitable $\Lambda$-ancestral selection graph.\\

		
		\\
	\end{abstract}

	\maketitle

	\section{Introduction}

	The determination of (asymptotic or approximate) formulae for the fixation probabilities of \blue{advantageous}
	alleles is a classical problem in population genetics, see e.g.\ \cite{PW08} for an overview. Seminal formulae have been established by Haldane \cite{H27} and by Kimura \cite{K62}. Both formulae have been shown to be valid for certain population models with selection and offspring distributions in the domain of attraction of the Wright-Fisher diffusion, hence by duality, 
	their genealogies under neutrality converge to Kingman's coalescent, see e.g.\ \cite{K62} and \cite{BoeGoPoWa1, BoeGoPoWa2}.
	
	Here we consider a two-type \blue{(wildtype and advantageous type)} population model \blue{with fixed population size $N$} and a skewed offspring distribution at neutral reproductions, which lies in the domain of attraction of a $\Lambda$-Wright-Fisher diffusion. To this we add moderate selection, i.e.\ genetic drift is considerably weaker than the effect of selection. In our setting the corresponding dual ancestral selection process has an asymptotically Beta$(2-\alpha, \alpha)$-like coalescence structure with $1 < \alpha < 2$ and an asymptotically linear, binary branching component.
	We show that the fixation probability $\pi_N$ of the advantageous allele \blue{when starting from a single copy} is asymptotically equal to  $\alpha^{1/(\alpha-1)} s_N^{1/(\alpha-1)}$, where $s_N \sim N^{-b}$ is the selection strength of the advantageous allele with $0< b<\alpha-1$.
	
	A similar scenario has also been investigated in \cite{Okada2021}. Okada and Hallatschek derive in their paper basically the same asymptotic formula (up to a constant) based on heuristic arguments. Here we give a rigorous proof of this asymptotics including the prefactor.
	Our proof is based on a duality argument, a method that has been proven to be useful for this type of question, see e.g. \cite{Krone1997}, \cite{M09},  \cite{BoeGoPoWa1}.
	
	The investigation of the interplay between selection and skewed offspring distributions is  biologically highly relevant as well as mathematically intriguing.   Skewed offspring distributions are believed to be common especially among marine populations and to play ``a major role in shaping marine biodiversity'' \cite{HedgecockPudovkin2011}, p. 971. Furthermore, also for rapidly adapting populations as well as for populations subject to seasonal reproduction cycles it has been argued that multiple merger coalescents provide a reasonable null-model for the genealogy of the population, see e.g. \cite{NH13} and \cite{CGSW22}. Empirical evidence for the relevance of multiple merger coalescents has been found in genomic data from various species, e.g.\  
	\emph{Atlantic cod} \cite{AKHE23},   
	\emph{Japanese sardine} \cite{NNY16}, 
	\emph{Mycobacterium tuberculosis} 
	\cite{MGF20} and \emph{Influenza} 
	\cite{SL12}. 
	
	Neutral population models with such genealogies have been intensively studied, see e.g.\ \cite{BBCEMSW05}, \cite{HM13}, \cite{HM22}, \cite{S03}, \cite{KW21, Freund21, BB21}. 
	Models that investigate the joint effect of selection and skewed offspring distributions have been analysed by various authors.
	One of the first models in this context was proposed and  investigated, in particular with respect to duality, by Etheridge, Griffiths and Taylor in \cite{EGT10}.
	Foucard \cite{F13}, Griffiths \cite{G14} as well as Bah and Pardoux \cite{BP15} consider $\Lambda$-Wright-Fisher models and $\Lambda$-lookdown models, respectively. They derive criteria for which fixation of the wildtype allele and  \blue{ of the advantageous}
	allele, respectively, is possible.  
	In \cite{CHV23} Cordero, Hummel and V\'{e}chambre study a class of $\Lambda$-Wright–Fisher processes with frequency-dependent and environmental selection. They show that in this class of models fixation of one type is not always sure, but coexistence of both types is possible.

	In \cite{GKP23} a general population model is analysed, in which individuals can reproduce neutrally and selectively according to skewed offspring distributions. In particular, the authors derive a semi-explicit formula for the fixation probability of a beneficial allele.
	Birkner, Dahmer and Eldon  consider in \cite{BDE23+} a related population model with selection and skewed offspring distribution evolving in discrete generations. It is a weighted version of the reproduction model in \cite{S03}, which involves free production of juveniles followed by population size regulation via (weighted) sampling without replacement from the juvenile pool. They arrive at the same asymptotic formula for the fixation probability as we do (with a  constant specific to their model).
	\blue{See also Section~\ref{Related work} below for more discussion.}
	
	\subsection{Population model and main result}
	We consider a population of fixed size $N \in \mathbb{N}$, \blue{where each individual is either of  wildtype or of the advantageous type}. The population is evolving in continuous time according to $\Lambda$-type neutral reproduction events as well as binary selective events. The selective and neutral events are governed by Poisson point processes $P_N$ and $P_S$. Let $\Lambda$ be a finite measure on $(0,1)$ and $c_N>0$. We start with the neutral component of the model. We label the $N$ individuals by ${1,\dots,N}$ and let  $P_N$ be a Poisson point process with intensity $c_N dt \otimes p^{-2} \Lambda(dp)$ on $\R_+ \times (0,1]$. For every point $(t,p)$ in $P_N$ each individual $i$  at time $t$ performs an independent coin toss with success probability $p$. All individuals with a successful coin toss participate in the reproduction event. Among the participating individuals a single individual is chosen uniformly at random to reproduce and replace with its offspring all other participating individuals. Note that this includes the possibility that only one or no individual participates in a reproduction event, resulting in `silent' events which have no effect on the frequency nor on the genealogy. \blue{Non-silent events occur at a finite rate, see \eqref{eq:neutral1} below, so the dynamics is well defined.} We will refer to the events associated to $P_N$ as \emph{neutral events}. Note that the neutral component of the model alone (obtained by setting $s_N=0$ below) is a variation of the `general Moran model' considered in \cite[Sect.~1.2.3]{BB09} and is also a time-continuous version of a special case of the class of models considered in \cite{HM13}, see also the discussion in \blue{ Subsection~\ref{sec:neutral}} below.
	
	In order to add selection we define $P_S$ to be a Poisson point process on $\R_+$ with intensity $N s_N dt$. For each point in $P_S$, we sample uniformly at random an individual $i$. Whenever $i$ is of the \blue{wildtype} %disadvantageous type 
	nothing happens, whenever $i$ is of the \blue{advantageous} %beneficial 
	type, $i$ reproduces in the sense that an uniformly chosen individual $j$ is replaced by a child of individual $i$, in particular it receives the (\blue{advantageous}) type of $i$. Similarly, we will refer to the events induced by $P_S$ as \emph{selective events}.

	\begin{figure}[h]
		\begin{center}
			\begin{tikzpicture}[x=1mm,scale=1.2,>=stealth]
				\foreach \x in {0,1,2,3,4} \draw[-] (-40,\x) -- (40,\x);
				\foreach \x in {1,2,3,4,5} \node[left] at (-40,\x-1) {\x};
				\draw [->,thick] (-30,2) -- (-30,3);
				\draw [->,thick, bend angle=30, bend left, blue] (-30,2) to (-30,4);
				\draw[-,thick, blue] (-30,4) -- (-17,4);
				\draw [->,thick] (-25,2) -- (-25,1);
				\draw [->,thick,  bend angle=20, bend right] (-17,0) to (-17,4);
				\draw [->,thick, bend angle=45, bend left] (-17,0) to (-17,2);
				\draw [->,thick, bend angle=30, bend left] (-17,0) to (-17,1);
				\draw [->,thick, bend angle=20, bend right] (-7,4) to (-7,2);
				\draw [->,thick, bend angle=20, bend left] (-7,4) to (-7,1);
				\draw [->,thick] (0,0) -- (0,1);
				%\draw [->,thick] (12,1) -- (12,2);
				%\draw [->,thick] (23,3) -- (23,4);
				\draw [->,dashed,thick] (-35,4) -- (-35,2);
				%		\draw [->,dashed,thick] (30,3) -- (30,2);
				\draw [->,dashed,thick] (15,0) -- (15,2);
				\draw[-,thick, blue] (-40,2) -- (-35,2);
				\draw[-,blue, thick] (-40,2) -- (-30,2);
				\draw[-,blue, thick] (-30,3) -- (-25,3);
				\draw[-,blue, thick] (-30,2) -- (-25,2);
				\draw [->,blue,thick] (-30,2) -- (-30,3);
				\foreach \x in {1,2,3}	\draw[-,blue,thick] (-25,\x) -- (-17, \x);
				\draw [-,blue,thick] (-17,3) -- (-7,3);
				\draw [->,blue,thick] (-25,2) -- (-25,1);
				\draw [->,blue,thick, bend angle=20, bend left] (23,3) to (23,4);
				\draw [->,blue,thick, bend angle=20, bend right] (23,3) to (23,2);
				%\draw[-,blue,thick] (-7,1) -- (0, 1);
				\draw[-,blue,thick] (-7,3) -- (30, 3);
				\draw[-,blue,thick] (23,4) -- (30,4);
				\foreach \x in {1,3,4}	\draw[-,blue,thick] (30,\x) -- (40, \x);
				\draw [->,dashed,blue, thick] (30,3) -- (30,1);
				\draw [-,thick] (-35,-0.6) -- (-35,-0.4);
				\node at (-35, -0.8){$t_0$};
				\draw [-,thick] (30,-0.6) -- (30,-0.4);
				\draw [-,thick] (-17,2) -- (23,2);
				\draw [-,thick, blue] (23,2) -- (40,2);  
				\node at (30, -0.8){$t_1$};
				\draw [->,thick] (-45,-0.5) -- (45,-0.5);
			\end{tikzpicture}
		\end{center}
		\caption{Selective arrows are depicted as dashed arrows. Individual~$3$ is of the \blue{advantageous} %beneficial 
			type, shown in blue. Hence, individual $3$ is not replaced at time $t_0$ since individual $5$ is of wildtype. At time $t_1$ the \blue{advantageous} %beneficial 
			descendant (individual $4$) of individual $3$ is able to use the selective arrow.}
		\label{Figure Moran 2}
	\end{figure}
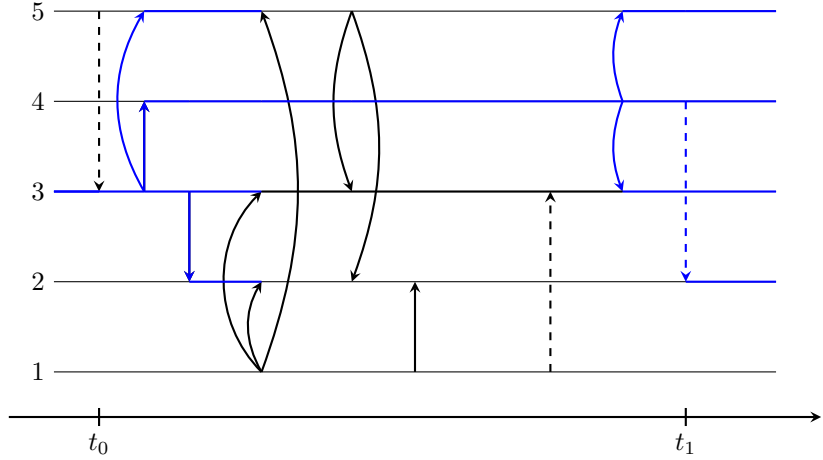
	
	In the following we  will consider $1<\alpha<2$ and measures $\Lambda$ of the form
	\begin{equation}
		\label{Beta(2-alpha,alpha)}
		\Lambda(dp) = \frac{1}{\Gamma(2-\alpha)\Gamma(\alpha)} p^{1-\alpha}(1-p)^{\alpha-1} \1_{(0,1)}(p) dp,
	\end{equation}
	i.e.\ $\Lambda = \mathrm{Beta}(2-\alpha,\alpha)$ and 
	\begin{equation}
		\label{choice c_N}
		c_N\sim (\alpha-1)\Gamma(\alpha) N^{1-\alpha}, \quad \text{as } N \to \infty,
	\end{equation}
	where $f_N\sim g_N$ is the shorthand notation for $f_N/g_N \to 1 $ as $N \to \infty$. The precise choice of scaling for $c_N$ is explained in \blue{Subsection~\ref{sec:neutral}}. Briefly speaking, $c_N$ is chosen such that each individual dies asymptotically at rate 1. \blue{The choice \eqref{Beta(2-alpha,alpha)} for $\Lambda$ is  prototypical in the context of offspring laws with infinite variance (see also \eqref{p_k tail} in Section~\ref{sec:neutral}) and makes explicit calculations easier, we refer to Section~\ref{sec:outlook} for background and generalisations.}
	
	From now on we consider a regime of \emph{moderate selection}, i.e. we assume
	\begin{align}\label{conditions_s_b_alpha}
		s_N \sim N^{-b},
	\end{align}
	for some $0<b<\alpha-1$, see \blue{Subsection~\ref{sec: moderate selection} } below for interpretation.
	
	Let $(X^{(N)}_t)_{t\geq 0}$ denote the number of wildtype individuals in the just defined model at time $t$ and set
	\begin{align}
		\pi_N := \lim_{t \to \infty} \PP_{N-1}(X^{(N)}_t = 0)
	\end{align}
	the probability of fixation of a single \blue{advantageous} %beneficial 
	mutant \blue{(here, $\PP_{N-1}$ refers to starting from $X^{(N)}_0=N-1$)}. We are now ready to state our main theorem.
	
	\begin{thm}\label{Theorem continous}
		Assume $0< b < \alpha-1$ and let $\Lambda(dx)$ be a Beta distribution with parameters $(2-\alpha,\alpha)$ for some $\alpha \in (1,2)$.  Then it holds that
		\begin{align} \label{eq:survival prob}
			\pi_N \sim \alpha^{\frac{1}{\alpha-1}} N^{\frac{-b}{\alpha-1}} = \left( \alpha s_N\right)^{\frac{1}{\alpha-1}} , \text{ as } N \to \infty.  
		\end{align}
	\end{thm}
	
	We give a sketch of the proof of Theorem~\ref{Theorem continous} \blue {in Section~\ref{sec:sketch of proof}} and rigorously prove it in \blue{Section~\ref{Sec:Proof}}. Furthermore, we discuss how our result connects to the fixation probabilities in the boundary cases $\alpha=2$ and $\alpha=1$ in \blue{Section~\ref{sec:outlook}}. 
	
	\begin{rem}\label{Remark:FixationProbability}
		\begin{itemize} 
			\item One could easily modify the selective strength in \eqref{conditions_s_b_alpha} by a constant $c_\mathrm{sel}$, i.e.\ use $s_N \sim c_\mathrm{sel} N^{-b}$, and obtain the same asymptotic behaviour for $\pi_N$ in \eqref{eq:survival prob} with $(\alpha c_\mathrm{sel} )^{\frac{1}{\alpha-1}}$ as the new constant.
			\item Note that smaller $\alpha$ reduces the probability of fixation. This is very plausible, since for smaller $\alpha$ neutral fluctuations in the frequency path are stronger.
		\end{itemize}
	\end{rem}

	\subsection{Choices of parameter scales}
	
	\subsubsection{Time scale and the size of neutral events} {\label{sec:neutral}}
	Consider a Poisson point process $\widetilde{P}_N$ on $\R_+ \times (0,1]$
	with intensity measure $dt \otimes p^{-2} \Lambda(dp)$ (which coincides with $P_N$ up to a
	re-scaling of time), and use the points in $\widetilde{P}_N$ to generate neutral reproduction
	events as above. Then, for a fixed individual among the $N$ many, the rate at which it participates in some non-trivial event (i.e., its own $p$-coin toss is successful and there is at least one other individual whose $p$-coin toss is successful) is given by
	\begin{align}
		\label{eq:neutral1}
		\widetilde{c}_N =
		\int_{(0,1]} p \big( 1 - (1-p)^{N-1}\big) \frac{1}{p^2} \Lambda(dp)
		\sim \frac{1}{(\alpha-1) \Gamma(\alpha)} N^{\alpha-1},
		\quad \text{as } N\to\infty .
	\end{align}
	Hence, the prefactor $c_N$ $(\sim 1/\widetilde{c}_N$, which is of order  $N^{1-\alpha}$) in the intensity measure of the
	Poisson point process $P_N$ ensures that any given individual dies due to
	a `neutral event' with (an approximately) constant rate,
	irrespective of the population size $N \gg 1$.
	
	Furthermore, given that the focal individual participates in the
	next neutral event, the conditional probability that the focal
	individual becomes therein the parent of $k \in \{2,3,\dots,N\}$
	offspring is given by
	\begin{align}
		\label{eq:neutral.k}
		& \widetilde{p}_k^{(N)} := \frac{1}{\widetilde{c}_N} 
		\frac{1}{k} \binom{N-1}{k-1} \int_{(0,1]} p \, p^{k-1} (1-p)^{N-k}\frac{1}{p^2} \Lambda(dp) \notag \\
		& \quad = \frac{1}{\widetilde{c}_N}
		\frac{1}{\Gamma(2-\alpha)\Gamma(\alpha)} \frac{\Gamma(N-k+\alpha)}{\Gamma(N-k+1)}
		\frac{\Gamma(k-\alpha)}{\Gamma(k+1)}
		\: \longrightarrow \: \frac{\alpha-1}{\Gamma(2-\alpha)} \frac{\Gamma(k-\alpha)}{\Gamma(k+1)}
		=: \widetilde{p}_k,
		\quad \text{as } N\to\infty,
	\end{align}
	and the conditional probability that the focal individual is not
	chosen as the parent and hence leaves $0$ offspring is given by
	\begin{align}
		\label{eq:neutral.0}
		& \widetilde{p}_0^{(N)} := \frac{1}{\widetilde{c}_N} \sum_{k=2}^N 
		\frac{k-1}{k} \binom{N-1}{k-1} \int_{(0,1]} p \, p^{k-1} (1-p)^{N-k}\frac{1}{p^2} \Lambda(dp)
		\: \longrightarrow \: \frac{1}{\alpha} =: \widetilde{p}_0 \quad \text{as } N\to\infty.
	\end{align}
	We refer to Section~\ref{sect:proofs.rem:neutral} for proofs of
	these formulas.
	
	Thus, the population model we consider here has the following
	equivalent description (in law): For $k=2,3,\dots,N$, each
	individual dies and is replaced by $k$ children with rate
	$\blue{c_N} \widetilde{c}_N \widetilde{p}_k^{(N)}$
	($\sim \widetilde{p}_k$ \blue{as $N\to\infty$}); in order to make
	room for these children, $k-1$ other individuals are drawn
	uniformly without replacement and die as well. In addition, each
	individual of the \blue{advantageous} %beneficial 
	type produces a single child at rate
	$s_N$ which replaces one uniformly chosen individual (possibly the
	parent itself).
	\smallskip
	
	We remark that the probability law with weights
	$(\widetilde{p}_k)_{k \in \N_0}$ (with $\widetilde{p}_1 := 0$)
	given in \eqref{eq:neutral.k}--\eqref{eq:neutral.0} has generating
	function
	\begin{align}
		\label{eq:genfcttilde}
		\widetilde{f}(z) = \frac{1}{\alpha}
		+ \sum_{k=2}^\infty \frac{\alpha-1}{\Gamma(2-\alpha)} \frac{\Gamma(k-\alpha)}{\Gamma(k+1)} z^k
		= \frac{1}{\alpha} (1-z)^\alpha + z,
	\end{align}
	where we used
	$$(1-z)^\alpha = 1 - \alpha z + \frac{\alpha(\alpha-1)}{\Gamma(2-\alpha)}
	\sum_{k=2}^\infty \frac{\Gamma(k-\alpha)}{\Gamma(k+1)} z^k.$$ In
	particular its mean is $\widetilde{f}'(1)=1$; we see e.g.\ from the
	explicit formula in \eqref{eq:neutral.k} that
	\begin{equation}
		\label{p_k tail}
		p_k \sim ((\alpha-1)/\Gamma(2-\alpha)) k^{-1-\alpha}, 
		\quad \text{as $k\to\infty$}
	\end{equation} so that it lies in the domain of attraction of a
	one-sided stable law with index $\alpha$, in particular, the
	variance is infinite.  Let us mention that the appearance of the law
	\eqref{eq:genfcttilde}, or a close relative like a version
	conditioned to be positive, is not at all surprising in the context
	of $\mathrm{Beta}(2-\alpha,\alpha)$-coalescents, see e.g.\
	\cite{K12, GY07}.
	
	\subsubsection{Selection of moderate strength} \label{sec: moderate selection}
	We regard the selection strength \eqref{conditions_s_b_alpha} as moderate, because for these values of $b$ we have $c_N \ll s_N\ll 1$. If $b=0$, selection is independent of $N$. This regime is generally called a regime of strong selection. When $c_N \in \Theta(s_N)$\blue{, i.e.\ $s_N \in O(c_N)$ and $c_N \in O(s_N)$ both hold,} selection and coalescence act on the same time scale. In the context of a Wright-Fisher diffusion, i.e. in the case $c_N = \frac{1}{N}$ and $s_N = \frac{\gamma}{N}$ for some $\gamma>0$, one then usually speaks of weak selection.  As the selection strength we are considering lies in between these two regimes, it is reasonable to speak of moderate selection.

	The case of skewed offspring distributions with weak selection, where in the scaling limit a $\Lambda$-Wright-Fisher process arises, has been studied     by various authors, see e.g. \cite{G14, F13, BP15}.
	
	\subsection{Duality}
	A key ingredient of the proof of the main result is the definition of an ancestral selection process $(A_\tau^{(N)})_{\tau\geq 0}$ that is dual to $(X_t^{(N)})_{t\geq 0}$. We express the fixation probability in terms of the expectation of the stationary distribution of this ancestral process. Hence, an analysis of this expectation yields the proof of Theorem \ref{Theorem continous}.

	The $\Lambda$-ancestral selection process $(A_\tau^{(N)})_{\tau \geq 0}$ counts the number of potential ancestors at (backward) time $\tau$ of a sample taken at time 0. It has the same transition rates as a suitably time rescaled block counting process of a $\Lambda$-coalescent with an additional branching mechanism. 
	\begin{defi}[$\Lambda$-ASP]\label{Def ASG}
		We define $A^{(N)}:=(A_\tau^{(N)})_{\tau \geq 0}$ with $A_0=n \in [N]$ to be the Markov process with state space $[N] \blue{:=\{1,\dots,N\} }$ and transition rates $r_{n,m}$ from $n$ to $m$:
		\begin{align} \label{eq:transition ASP}
			r_{n,n+1} &= n s_N (1- \frac{n}{N}), \qquad &n \in \{1,...,N-1\}\\
			r_{n, j} &= c_N q(n,j),  \, &j  \in \{1,...,n-1\}
		\end{align}
		with $c_N  \sim  (\alpha-1) \Gamma(\alpha) N^{1-\alpha} $ and 
		\begin{equation}\label{jump rates q}
			q(n,j) = \binom{n}{n-j+1} \int_{0}^1 x^{n-j-1} (1-x)^{j-1} \Lambda (dx).
		\end{equation}
		
	\end{defi}
	
	The $\Lambda$-ASP is the block counting process corresponding to the $\Lambda$-ancestral selection graph that one obtains informally by reversing time in Figure~\ref{Figure Moran 2}. \blue{Let us mention here the related works \cite{CV23, GKP23}, where analogously defined $\Lambda$-ancestral selection graphs in different settings are considered.}

	Note that the jump rates \blue{$r_{n,j}$} from $n$ to $j$ of the $\Lambda$-ASP are given by the respective time changed \blue{(through multiplication with $c_N$)} rates of the block counting process of the $\Lambda$-coalescent (i.e.\ the rate at which some $(n-j+1)$-merger occurs when there are presently $n$ blocks). The practical interest in defining the $\Lambda$-ASP lies in the following sampling duality. It has been used widely in the literature, in particular also to determine (asymptotic) expressions for the fixation probability. Basically it states that  all individuals in a sample have the wildtype iff all potential ancestors of these individuals have the wildtype.
	
	\begin{prop}[Duality] \label{thm:duality}
		Let $k,n \leq N$ and $t>0$. The following hypergeometric duality holds
		\begin{align}
			\E\left[\frac{X_t^{(N)}(X_t^{(N)}-1)\cdots(X_t^{(N)}-(n-1))}{N(N-1)\cdots(N-(n-1))}\Big | X_0^{(N)}=k\right] \notag \\ =\mathbb E\left[\frac{k(k-1)\cdots(k-(A^N_t-1))}{N(N-1)\cdots(N-(A_t^N-1))}\Big | A_0^N=n\right], \label{eq:duality}
		\end{align}
		where the transition rates of $A^{(N)}$ are given in Definition \ref{Def ASG}.
	\end{prop}
	The most relevant case  of Proposition $\ref{thm:duality}$ for our analysis is the case $k=N-1$, i.e. at time 0 in the population a single individual has the \blue{advantageous} %beneficial 
	type. A simple consequence of Proposition \ref{thm:duality} is the following corollary.
	\begin{cor}\label{cor:survival prob}
		Denote by $A_{eq}^{(N)}$ a random variable whose distribution is given as the stationary distribution of $A^{(N)}$. Then the following holds
		\begin{align}
			\pi_N = \frac{\EE{A_{eq}^{(N)}}}{N}. \label{eq:expression for the surv prob}
		\end{align}
	\end{cor}
	\begin{proof}
		Apply Proposition \ref{thm:duality} with $k=N- 1$, \blue{$n=N$ (noting that then $X_t^{(N)} (X_t^{(N)}-1)\cdots(X_t^{(N)}-(n-1)))=N!$ if $X^{(N)}_t=N$ and $0$ otherwise)} and make use of the fact that $A_t^{(N)} \to A_{eq}^{(N)}$ in distribution as $t \to \infty$ together with the dominated convergence theorem, then \eqref{eq:expression for the surv prob} follows.
	\end{proof}
	\begin{proof}[Proof of Proposition \ref{thm:duality}]
		Due to the construction of the $\Lambda$-ASG (see Def.~\ref{Def ASG} and the discussion below) in order to sample only wildtype individuals at time $t$ all ancestors have to be wildtype individuals. By exchangeability it is only important to keep track of the number of potential ancestors. Hence, in order to prove \eqref{eq:duality} it suffices to show that the rates given in the definition match to the rates at which we gain or lose potential ancestors. Assume that the $\Lambda$-ASP consists currently of $n$ lines.  Whenever there is a selective reproduction event (which occurs at rate $s_N N$), with probability $n/N (1-n/N)$ an individual from outside the current $n$ lines is chosen to reproduce with offspring among the $n$ lines, hence we have to follow an additional line and thus the $\Lambda$-ASP branches. This gives the rate $r_{n,n+1}$.
		
		Whenever there is a neutral reproduction event with intensity $p$ at rate $c_N$ all $n$ lines toss a coin, with $n-j+1$ participating with probability $\binom{n}{n-j+1}p^{n-j+1}(1-p)^{j-1}$. All the participating lines will merge resulting in a single line for the $\Lambda$-ASP. Hence we lose $n-j+1$ lines but another one is gained and we need to follow $j$ lines after the event. This gives $r_{n,j}$ for all $j \in {1,\dots,n-1}$.
	\end{proof}
	
	\subsection{Some related models} 
	\label{Related work}
	
	\paragraph{The $\Lambda$-asymmetric Moran model from
		\cite{GKP23}} As already mentioned in the introduction, in
	\cite{GKP23} a related class of models involving multiple merger
	genealogies and two selectively distinguished types is studied. As
	in our paper, the corresponding dual $\Lambda$-ancestral selection
	process plays an important role in \cite{GKP23}. There are however
	important differences. In the notation of \cite{GKP23}, there are
	$N$ individuals which can be of type $\oplus$ (the fitter type) or
	of type $\ominus$ (the less fit type). Reproduction is parametrised
	by two (probability) measures $\Lambda^{\oplus}$ and
	$\Lambda^{\ominus}$ on $[0,1]$, where $\Lambda^{\oplus}$
	stochastically dominates $\Lambda^{\ominus}$. At rate $1$, an
	individual is chosen uniformly at random (irrespective of its type)
	to reproduce; if the reproducing individual is of type $\oplus$, a
	value $y$ is drawn from $\Lambda^{\oplus}$.  Every individual then independently
	dies instantaneously with probability $y$ and is immediately
	replaced by an offspring of the reproducing individual (with type
	$\oplus$).  The procedure is analogous if the reproducing individual
	is of type $\ominus$ except that $y$ is drawn from
	$\Lambda^\ominus$. Thus, the difference in offspring laws between
	the two types can (and generically will be) much greater in this
	model than in our set-up, and the model from \cite{GKP23} gives
	greater freedom to model these differences (a measure on $[0,1]$
	rather than a single rate as in our case). Arguably, both selective
	and neutral reproduction in \cite{GKP23} are `event-based' rather
	than `individual-based': as it stands in \cite[Sect.~1--5]{GKP23},
	the individual offspring law in the model from \cite{GKP23} depends
	on the type and on the population size $N$, the individual litter size will
	always be a (random) non-trivial fraction of the total population
	size $N$. By contrast, our setup is chosen in such a way that the
	individual offspring laws stabilise to a proper probability law on
	$\N$, see \eqref{eq:neutral.k}--\eqref{eq:neutral.0} and the
	discussion around it. In particular, our model is not
	a special case of the model from \cite{GKP23}.
	
	Furthermore, the $\Lambda$-ASP in \cite{GKP23} differs from our
	process in two aspects.  The rate of downward jumps in \cite{GKP23}
	depends on whether the potential parent is among the potential
	ancestors or not. Due to our construction of the model we do not
	need to distinguish these cases, which simplifies the transition
	rates in our case. Furthermore, for our $\Lambda$-ASP the branching
	rate is asymptotically linear in $n$ (if $n\in o(N)$, which is at
	least under stationarity with high probability the case).  This is
	in the setting considered in \cite{GKP23} in general not the
	case. The question analogous to our Theorem~\ref{Theorem continous},
	namely the $N\to\infty$ asymptotics of the fixation probability of a
	single advantageous mutant in a population of size $N$ is not in the
	focus of \cite{GKP23}. (\cite{GKP23} do consider in Sections~6--7 an
	SDE version obtainable as a rescaling of their discrete population
	models and derive in Proposition~7.2 a series representation for the
	fixation probability starting from a fraction $x \in (0,1)$ of fit
	individuals in the limiting SDE model.  It is tempting to insert
	$x=1/N$ there, yet is at least a priori not clear in how far this
	reflects the behaviour of the pre-limiting models starting from a
	single fit individual.)
	
	Given the fact that \cite{GKP23} as well as the present paper use
	population models which are in sampling duality with certain
	branching-coalescing processes, it is conceivable that the techniques
	of the present paper could be applied to derive an analogue of
	Theorem~\ref{Theorem continous} for the model from \cite{GKP23}.  In
	fact, the basic proof structure we use for Theorem~\ref{Theorem
		continous} carries over verbatim and the analogue of
	\eqref{eq:expression for the surv prob} from
	Corollary~\ref{cor:survival prob} holds for the model from
	\cite{GKP23}.  Due to the more complicated structure of the
	transition rates of the ancestral process in \cite{GKP23} and the
	(much) greater generality in the choice of $\Lambda^\oplus$ and
	$\Lambda^\ominus$ compared to our set-up, a detailed analysis of the
	asymptotics of the fixation probabilities $\pi_N$ in the setting of
	\cite{GKP23} seems challenging and must presently be left for
	possible future work.
	
	A preliminary observation can already be made at this point:
	%%(we focus for simplicity on the case that $\Lambda^\oplus$ and $\Lambda^\ominus$
	%%are both probability measures):
	Write $\mu^\oplus = \int_{(0,1]} y \, \Lambda^\oplus(dy)$
	and $\mu^\ominus = \int_{(0,1]} y \, \Lambda^\ominus(dy)$
	for the means of $\Lambda^\oplus$ and $\Lambda^\ominus$, respectively.
	Inspection of the jump rates of the ancestral process in \cite[Prop.~4.5]{GKP23}
	(cf also the displayed equation below Proposition~4.5 there) shows that the
	expected loss rate of potential ancestors when there are currently $n$ of them
	is
	\begin{align*}
		& \frac{n}{N} \int_{[0,1]} \sum_{k=1}^{n-1} k \binom{n-1}{k} y^k(1- y)^{n-1-k} \, \Lambda^\ominus(dy) \\
		&\quad \quad \quad
		+ \Big(1-\frac{n}{N}\Big) \int_{[0,1]} \sum_{k=1}^{n-1} k \binom{n}{k+1} y^{k+1} (1-y)^{n-(k+1)} \, \Lambda^\ominus(dy) \\
		& \quad = \frac{n}{N} \int_{[0,1]} (n-1)y \, \Lambda^\ominus(dy) + \Big(1-\frac{n}{N}\Big) \int_{[0,1]} \sum_{j=2}^n (j-1) \binom{n}{j} y^j(1-y)^{n-j} \, \Lambda^\ominus(dy) \\
		& \quad = \frac{n}{N} (n-1)\mu^\ominus 
		+ \Big(1-\frac{n}{N}\Big) (n \mu^\ominus - 1 ) 
		+ \Big(1-\frac{n}{N}\Big) \int_{[0,1]} (1-y)^n \, \Lambda^\ominus(dy) \\
		& \quad = \Big(1-\frac{1}{N}\Big) n \mu^\ominus
		+ \Big(1-\frac{n}{N}\Big) \Big( \int_{[0,1]} (1-y)^n \, \Lambda^\ominus(dy) - 1 \Big) ,
	\end{align*} 
	while the expected gain rate is always bounded by the total mass of $\Lambda^\ominus$
	(which is $O(1)$ in \cite[Sect.~1--5]{GKP23}). Using the heuristic idea that
	the equilibrium point $\E[A^{(N)}_{eq}]$ should be close to that value of $n$
	where the total loss rate balances the total gain rate exactly (cf the argument at
	the beginning of Section~\ref{sec:sketch of proof}) thus suggests that
	$\E[A^{(N)}_{eq}] \approx O(1)$ and hence $\pi_N = O(1/N)$ in the setting from
	\cite[Sect.~1--5]{GKP23}.
	In fact, writing $\widetilde{L}^{(N)}$ for the generator of the ancestral process
	from \cite[Prop.~4.5]{GKP23} and using the function $g(x)= c x$ with a suitably chosen constant $c \in (0,\infty)$, quite analogous
	to the argument in Lemma~\ref{lm:upper} below, one finds
	\[
	\max_{x \in \{0,1,\dots,N\}} \widetilde{L}^{(N)}g(x) + x \le C,
	\]
	for some constant $C=C(\Lambda^\oplus,\Lambda^\ominus) < \infty$. Thus
	indeed $\E[A^{(N)}_{eq}] \le C$ and $\pi_N \le C/N$ by arguing as in the proof of Theorem~\ref{Theorem continous}.
	On the other hand, $\pi_N \ge 1/N$ is automatic and thus $\pi_N = \Theta(1/N)$ in this model
	whenever $\Lambda^\ominus \neq 0$.
	
	This appears also intuitively plausible: starting with one fit individual, this has a chance of $1/N$
	to be chosen as the parent in each reproduction event. If it is chosen, its offspring will
	afterwards constitute a non-trivial fraction of the population and since from then on the
	selective advantage can play out its full strength, it will be likely that the fit type prevails
	henceforth. However, if it not chosen as a parent, it has a non-trivial chance to die and be replaced
	by an offspring of an unfit individual. Thus, the chance that the initial fit individual
	becomes a parent before dying is $O(1/N)$.

	\paragraph{A two-type version Schweinsberg's model \cite{S03} with selection}
	In \cite{BDE23+}, the authors prove the asymptotics of the fixation probability in a moderate selection regime in a related model which uses discrete generations and weighted resampling from a pool of juveniles, extending \cite{S03} who studied the neutral case.  
	The treatment in \cite{BDE23+} follows the classical ``forwards in time'' route: a branching process approximation in the early stage of establishment (and in the final stage) together with quantitative control on the deviation from a deterministic approximation of the frequency path in the intermediate stage.
	Indeed, consider the fate of an individual of the \blue{advantageous} %beneficial 
	type in our
	model at a time when that type is rare: Heuristically,
	there should then be almost no interaction among the \blue{advantageous} %beneficial 
	individuals, which suggests to make the limiting law with
	generating function \eqref{eq:genfcttilde} slightly supercritical
	by allowing additional duplications at rate $\widetilde{c}_N s_N$
	in addition to neutral births/deaths at constant rate, i.e.\
	consider
	\begin{align}
		\label{eq:genfNtilde}
		\widetilde{f}_N(z) = s_N(1+s_N)^{-1} z^2 + (1+s_N)^{-1} \widetilde{f}(z).
	\end{align}
	In fact, one can check that then the survival probability
	$1-\widetilde{q}_N = 1-\widetilde{f}_N(q_N) > 0$ of a Galton-Watson
	process with \blue{reproduction} law \eqref{eq:genfNtilde} satisfies the asymptotics
	$1-\widetilde{q}_N \sim \mathrm{const.} \times
	(s_N)^{1/(\alpha-1)}$ as $N\to\infty$, which fits well to Theorem~\ref{Theorem
		continous}.  Details of such computations can be found in \cite{BDE23+}.
	
	\subsection{Sketch of the proof of Theorem \ref{Theorem continous}}
	\label{sec:sketch of proof}

	Our proof is based on the duality between the process $X^{(N)}$ that records the number of wildtype individuals and the ancestral selection process $A^{(N)}$. Due to Corollary \ref{cor:survival prob} the statement is proven once we show that $\EE{A^{(N)}_{eq}} \sim \alpha^{\frac{1}{\blue {\alpha-1}}} N^{1 -\frac{b}{\alpha-1}}$. 
	Heuristically, the expectation should be asymptotically equal to $\alpha^{\frac{1}{\blue{\alpha-1}}} N^{1 -\frac{b}{\alpha-1}}$, since at this point upwards and downwards drift cancel. Indeed, in Lemma \ref{Lem exact form of the generator} below we show that in state $x$ the upwards drift is  $\sim x s_N$ and the downwards drift is   $\sim c_N x^{\alpha} \frac{1}{(\alpha-1) \Gamma(\alpha+1)} $. So upwards and downwards cancel (asymptotically), if $x N^{-b} = c_N x^{\alpha}\frac{1}{(\alpha-1) \Gamma(\alpha+1)}$. This is the case for  $x = \left(\alpha N^{\blue{\alpha -1-b}}\right)^{\frac{1}{\blue{\alpha-1}}}.$ 
	
	To control the expectation of $\EE{A^{(N)}_{eq}}$ rigorously we use a Lyapunov-type argument, see \cite{GlynnZeevi2008} for a justification of this argument in more general settings  (with infinite state spaces).
	Denote the rate matrix of $A^{(N)}$ by $Q^{(N)}$ and assume the entries of $\pi^{(N)}:=(\pi_1^{(N)}, ..., \pi_N^{(N)})$ are the probability weights of the stationary distribution of $A^{(N)}$ in $1, ..., N$. By definition we have $\pi^{(N)} Q^{(N)}=0$ and since the state space of $A^{(N)}$ is finite we also have for any vector $g^{(N)} ={(g^{(N)}_1, ..., g^{(N)}_N)}^T \in \R^{N}$ that $\pi^{(N)} Q^{(N)} g^{(N)} =0.$ Consequently, if we find a constant $c^{(N)}$ and a vector $g^{(N)}$, such that 
	\begin{align}\label{ConditionUpperBound}
		Q^{(N)} g^{(N)} + (1, ..., N)^T \leq  c^{(N)} \cdot(1, ..., 1)^T,\end{align} we have 
	\begin{align*}
		\EE{A_{eq}^{(N)}} = \pi^{(N)} \cdot (1, ..., N)^T \leq c^{(N)}.
	\end{align*} 
	
	The key step in the proof of Theorem \ref{Theorem continous} is to find a vector $g^{(N)}$
	(which also can be interpreted as a function $g^{(N)}:\{1, ..., N\} \rightarrow \R$) and a constant $c^{(N)}\leq  \blue{\alpha^{\frac{1}{\alpha-1}}} N^{1 -\frac{b}{\alpha-1}} (1+o(1)) $,
	such that \eqref{ConditionUpperBound} is fulfilled. We will see that for the upper bound it suffices to consider linear functions, i.e. functions of the form $g^{(N)}(x) = a x$ with $a \in \R$, and determine an appropriate  $a = a_N$.
	This also corresponds to the heuristic calculation that $\EE{A_{eq}^{(N)}}$ is (close to) the point where the upwards and the downwards parts of the drift balance.

	One argues analogously for a lower bound on $\EE{A_{eq}^{(N)}}$. In this case linear functions are not sufficient, we consider instead functions of the form $g^{(N)}(x) = a_1 x^{\beta_1} + a_2 x^{\beta_2}$ with $0 <\beta_1 <\beta_2 <\alpha-1$ and constants $a_1= a_1(N)$ and $a_2= a_2(N)$.

	\subsection{Extensions}\label{sec:outlook}
	
	\subsubsection*{Fixation probabilities in the boundary cases $\alpha=2$ and $\alpha=1$}
	In Theorem \ref{Theorem continous} we consider the case $\alpha\in (1,2)$. The boundary case $\alpha=2$ corresponds to the parameter regime in which the neutral genealogies lie in the domain of attraction of the Kingman coalescent, and in the boundary case $\alpha=1$ the corresponding neutral genealogies are in the domain of attraction of the Bolthausen-Sznitman coalescent. 
	
	\begin{itemize}
		\item \blue{ For the case $\alpha =2$, note that Beta$(2-\alpha,\alpha) \to \delta_0$ as $\alpha\to 2$. For $\Lambda= \delta_0$ there are only binary neutral events and hence our model corresponds to the classical Moran model with selection.} It has been shown recently \blue{in a more general setting within the domain of attraction of Kingman's coalescent} that the fixation probability of a moderately advantageous allele is asymptotically equal to $\frac{2 s_N}{\sigma^2}$, where $\sigma^2$ denotes the individual offspring variance, see \cite{BoeGoPoWa1} and \cite{BoeGoPoWa2}. This formula goes back to the well-known approximations given by Haldane  \cite{H27} and by Kolmogorov \cite{Kolmogorov1938}. Letting $\alpha\rightarrow 2$ in Theorem \ref{Theorem continous} yields $\pi_N \sim 2 s_N$, that is we get back Haldane's asymptotic.
		
		In the case of a Moran model in the attraction of Kingman's coalescent  ($\Lambda=\delta_0$), the stationary distribution of $A_{eq}^{(N)}$ is known to be a binomial random variable  with parameters $N$ and $\frac{2 s_N}{ 2 s_N +1}$ conditioned to be larger than $0$ as for example observed in \cite{BoeGoPoWa1} or \cite{Cordero2017}. Applying \eqref{eq:expression for the surv prob} then yields Haldane's asymptotic.

		\item The case $\alpha=1$, i.e.\ in the domain of attraction of a Bolthausen-Sznitman coalescent, has been considered in \cite{H18}.   
		In this regime the pair coalescence probability $c_N$ fulfils $c_N\sim\frac{1}{\log(N)}$ and hence for a regime of moderate selection we suggest to consider $$s_N = \frac{1}{(\log N)^b},$$
		with $0<b<1$, see Subsection~\ref{sec: moderate selection}. 
		\blue{An ancestral selection graph which is analogous to the one from Definition~\ref{Def ASG} for the case $\alpha=1$ has transition rates $\tilde{r}_{n,m}$ from $n$ to $m$ of the following form: }
		\begin{align} 
			\tilde{r}_{n,n+1} &= n s_N (1- \frac{n}{N}), \qquad &n \in \{1,...,N-1\}\\
			\tilde{r}_{n, j} &= c_N \frac{n}{(n-j+1)(n-j)} ,  \, &j  \in \{1,...,n-1\}.
		\end{align}
		Hence, the upwards drift is $\sim n s_N$ and since  $\sum_{j=1}^{n-1} \frac{n-j}{(n-j+1)(n-j)} \sim  \log(n)$ the downwards drift $\sim c_N n \log(n) = \frac{n \log(n)}{\log N}.$ Upwards and downwards drift cancel if $$n = \exp\left( s_N \log N\right).$$ This suggests that the probability of fixation of an allele with selective advantage $s_N= \left(\frac{1}{\log(N)}\right)^b$ is asymptotically 
		\begin{align}\label{FixBS}
			\tilde{\pi}_N \sim \frac{ \exp( \log(N)^{1-b})}{N} = N^{s_N -1}.
		\end{align}
		This coincides with (the heuristic) formula (16) in \cite{H18} for the fixation probability of an advantageous allele when $s= \frac{1}{\log(N)^b}$ and $x_0 = \frac{1}{N}$.  
		One should be able to prove rigorously the asymptotic \eqref{FixBS}  with a similar approach as in the proof of Theorem \ref{Theorem continous}. 
	\end{itemize}
	
	\subsubsection*{Other choices of $\Lambda$}
	In Theorem \ref{Theorem continous} we assume that $\Lambda(dx)$ is exactly given as a Beta-distribution with parameters $(2-\alpha,\alpha)$, which makes the calculations in Section~\ref{Sec:Proof} more explicit. However, we believe that only the singularity of $\Lambda$ at $0$ is crucial for proving \eqref{eq:survival prob}. In particular, as long as $\Lambda(dx)$ is such that one obtains the same leading order in Lemma~\ref{Lem exact form of the generator} and one can prove Lemma~\ref{lem:Coalescence Generator}, the asymptotics for the survival probability will follow.
	In view of the computations in \cite{GY07}, we believe that our approach can be carried through if
	instead of the explicit form \eqref{Beta(2-alpha,alpha)} of $\Lambda$ we only assume that
	$\Lambda([0,x]) = c x^{2-\alpha} + O(x^{3-\alpha})$ as $x \downarrow 0$ with some constant $c \in (0,\infty)$.
	Undoubtedly, this extension will make the calculations more tedious, which is why we refrain from doing so here.

	\section{Proof of Theorem \ref{Theorem continous}} \label{Sec:Proof}
	
	Denote by $L^{(N)}$ %$L^{(N)}:C_b(\R)\to C_b(\R)$ 
	the generator of the $\Lambda$ ancestral selection process $A^{(N)}$ from Definition~\ref{Def ASG}. The key ingredient for the proof of Theorem \ref{Theorem continous} is the following Proposition.
	
	\begin{prop}\label{proposition}
		There exists  $c^{(N)}_1$ and $c^{(N)}_2 \in \mathbb R$ with $c^{(N)}_1, c^{(N)}_2 \sim \alpha^{\frac{1}{\blue{\alpha-1}}} N^{1- \frac b{\alpha-1}} $ as $N \to \infty$. %(and $c_\alpha$ as in Theorem \ref{Theorem continous}).
		Additionally, there exist functions $ g_1^{(N)}$ and $g_2^{(N)} : \R_+ \to \R$ fulfilling
		\begin{align}
			\sup_{x \in [N]} L^{(N)}g_1^{(N)}(x) + x &\leq c^{(N)}_1 \label{crucial inequality1} \\ 
			\inf_{x \in [N] } L^{(N)} g_2^{(N)}(x) +x &\geq c^{(N)}_2 \label{crucial inequality2},
		\end{align}
		for $N$ large enough.
	\end{prop}
	\blue{The proof of Proposition \ref{proposition} will be given at the end of Section \ref{Section: Proof Prop 1}.} Next we give (on the basis of the statement of Proposition \ref{proposition}) a proof of Theorem \ref{Theorem continous}.
	
	\begin{proof}[Proof of Theorem \ref{Theorem continous}]
		As pointed out in the sketch of the proof of Theorem \ref{Theorem continous}  Proposition \ref{proposition} yields (by writing the above inequalities in terms of the rate matrix $Q^{(N)}$ of $A^{(N)}$, see \eqref{ConditionUpperBound}) that there exist $c_1^{(N)}$ and $c_2^{(N)}$ with $c^{(N)}_1, c^{(N)}_2 \sim\alpha^{\frac{1}{\blue{\alpha-1}}}N^{1- \frac b{\alpha-1}}$, such that 
		\begin{align}\label{inequalities AEeq}
			c^{(N)}_2 \leq \EE{A^{(N)}_{eq}} \leq c^{(N)}_1.
		\end{align}
		Then the claim of Theorem \ref{Theorem continous} is a direct consequence of \eqref{inequalities AEeq} and Corollary \ref{cor:survival prob}.
	\end{proof}

	\subsection{Bounds on the generator}
	The proof of Proposition~\ref{proposition} relies on Lemma~\ref{lm:upper} and Lemma~\ref{lem:lower bound gen}, which will be proved in Section~\ref{Section: Proof Prop 1}. We prepare both proofs by gathering some results regarding the generator $L^{(N)}$ of the $\Lambda$-ancestral selection process.
	
	The generator $L^{(N)}$ of $A^{(N)}$ acts as follows on functions $g: \R_+ \to \R$, with $x \in [N]$
	\begin{align}
		L^{(N)}g(x)= x s_N (1 - \frac{x}{N})\left( g(x+1)-g(x) \right) + c_N \sum_{y=1}^{x-1} q(x,y) \left( g(y)-g(x) \right). \label{Generator continuous}
	\end{align}
	
	\begin{lem}\label{lem:jump rates block counting}
		For the transition rates $q(x,y), 1 \le y < x$  given in \eqref{jump rates q} it holds 
		\begin{align}
			\label{eq:q(x,y)}
			q(x,y)= \frac{x}{\Gamma(2-\alpha)\Gamma(\alpha)}
			\frac{\Gamma(y+\alpha-1)}{\Gamma(y)} \frac{\Gamma(x-y-\alpha+1)}{\Gamma(x-y+2)}.
		\end{align}
	\end{lem}
	
	The proof can be found (for completeness) in the appendix.

	\begin{rem}
		\label{rem:q(x,y).approx}
		Using Gautschi's inequality, \eqref{eq:q(x,y)} implies that there exists a constant $c=c(\alpha) \in (1,\infty)$
		such that \blue{uniformly} for all $x > y \ge 1$
		\begin{align}
			\label{eq:q(x,y).approx}
			\frac{1}{c} \le \frac{x y^{\alpha-1}(x-y)^{-\alpha-1}}{q(x,y)} \le c.
		\end{align}
	\end{rem}
	
	\begin{lem} \label{Lem exact form of the generator}
		For functions $g: \R_+ \to \R, g(x)= a x$ with $a=a(N) \in \R$ it holds for $x\in \{1, \dots, N\}$
		\begin{align*}
			&L^{(N)} g(x)=a s_N \left( 1 - \frac{x}{N}\right)x- a  c_N x\frac{  \Gamma(x+\alpha) - \Gamma(\alpha+1) \Gamma(x+1)}{ (\alpha -1 ) \Gamma(\alpha +1) \Gamma(x+1)}.
		\end{align*}
	\end{lem}
	\begin{proof}
		Recall \eqref{Generator continuous}. We use the following equality which follows directly from Equations (1), (5) and (8) in \cite{GY07} ($q(x,y)$ is $\lambda_{x,x-y+1}$ in their notation)
		\begin{align*}
			\sum_{y=1}^{x-1} (x-y) q(x,y) & = \sum_{y=1}^{x-1} (x-y) \int_0^1 (1-u)^{y-1} \Lambda(du)
			\\ 
			&= \sum_{y=1}^{x-1} (x-y) \int_{0}^{1} \frac{(1-u)^{y-1} u^{1-\alpha} (1-u)^{\alpha -1 } }{\Gamma(2-\alpha) \Gamma(\alpha)} du
			\\ &= \sum_{y=1}^{x-1} (x-y)\frac{\text{Beta}(2-\alpha,y-1+\alpha)}{\Gamma(2-\alpha) \Gamma(\alpha)} 
			\\ &= \sum_{y=1}^{x-1} (x-y)\frac{\Gamma(y-1+\alpha)}{\Gamma(y+1) \Gamma(\alpha)},
		\end{align*}
		where Beta$(a,b)$ denotes the Beta-function, i.e. Beta$(a,b)= \frac{\Gamma (a) \Gamma(b)}{\Gamma(a+b)}$ for $a,b, >0$.
		
		Together with an application of Lemma \ref{lem:Identities gamma function final 1} this yields
		\begin{align}
			\sum_{y=1}^{x-1} (x-y) q(x,y) & =  \frac{1}{\Gamma(\alpha)} \frac{ x \Gamma(x+\alpha) - x\Gamma(\alpha+1) \Gamma(x+1)}{\alpha (\alpha -1 ) \Gamma(x+1)} \notag \\
			&=\frac{ x \Gamma(x+\alpha) - x\Gamma(\alpha+1) \Gamma(x+1)}{ (\alpha -1 ) \Gamma(\alpha +1) \Gamma(x+1)}.
			\label{first moment gen}
		\end{align}
		This together with \eqref{Generator continuous} proves the lemma.
	\end{proof}

	\begin{lem}\label{lem:Coalescence Generator}
		Let $0 <\beta < 1$. %$2-\alpha$. 
		We have for some constant $c_{\alpha, \beta} \in (0, \infty)$
		\begin{align}
			\left| \sum_{y=1}^{x-1} q(x,y) (y^{\beta}-x^{\beta}) +\frac{\beta }{(\alpha-1) \Gamma(\alpha+1)} x^{\beta} \frac{\Gamma(x+\alpha)}{\Gamma(x+1)} \right| \leq c_{\alpha, \beta} x^{\beta}. \label{eq:lem bound line counting}
		\end{align}
	\end{lem}
	
	\begin{proof} 
		For $x > y \ge 1$, a Taylor expansion to first order with integral remainder term yields
		\begin{align}
			y^\beta - x^\beta
			& = -\beta x^{\beta-1}(x-y) - \beta(1-\beta) \widetilde{R}(x,y)
			\quad \text{with} \quad 
			\widetilde{R}(x,y) := \int_y^x (t-y) t^{\beta-2} \, dt .
		\end{align}
		We have
		\begin{align}
			0 \le \widetilde{R}(x,y) \le \int_0^x t^{\beta-1} \, dt = \frac{1}{\beta} x^\beta,
		\end{align}
		and in addition (for any $0 < \varepsilon <1$)
		\begin{align}
			\widetilde{R}(x,y)
			\le 
			\varepsilon^{\beta-2} x^{\beta-2} \frac{(x-y)^2}{2}, \quad \text{ if \ } \varepsilon x \le y < x.
		\end{align}
		Thus, using \eqref{first moment gen}, 
		\begin{align}
			\sum_{y=1}^{x-1} q(x,y) (y^{\beta}-x^{\beta})
			& \le -\beta x^{\beta-1} \sum_{y=1}^{x-1} q(x,y) (x-y)
			= -\frac{\beta}{(\alpha -1 ) \Gamma(\alpha +1)} x^\beta \frac{\Gamma(x+\alpha)}{\Gamma(x+1)}
			+ \frac{\beta}{\alpha -1} x^\beta .
		\end{align}
		For the other bound we first estimate, using Remark~\ref{rem:q(x,y).approx}
		in the first inequality, \blue{and pretending $\varepsilon x \in \N$ to ease readability},
		\begin{align}
			\sum_{y=1}^{x-1} q(x,y) \widetilde{R}(x,y)
			& \le C x \sum_{y=1}^{x-1} y^{\alpha-1} (x-y)^{-\alpha-1} \widetilde{R}(x,y) \notag \\
			& \le C x \sum_{y=1}^{\varepsilon x} y^{\alpha-1} (x-y)^{-\alpha-1} x^\beta +
			C x \sum_{y=\varepsilon x}^{x-1} y^{\alpha-1} (x-y)^{-\alpha-1} x^{\beta-2} (x-y)^2 \notag \\
			& \le C x^{\beta-\alpha} \sum_{y=1}^{\varepsilon x} y^{\alpha-1}
			+ C x^{\beta} \sum_{y=\varepsilon x}^{x-1} \frac{1}{x} \Big( 1 - \frac{y}{x} \Big)^{1-\alpha} \Big( \frac{y}{x} \Big)^{\alpha-1}
			\le C x^{\beta},
		\end{align}
		since
		\begin{align}
			\sum_{y=1}^{\varepsilon x} y^{\alpha-1} &\le \int_0^{\varepsilon x+1} t^{\alpha-1} \, dt \le C x^\alpha,
			\quad \text{ and } \\
			\quad \sum_{y=\varepsilon x}^{x-1} \frac{1}{x} \Big( 1 - \frac{y}{x} \Big)^{1-\alpha} \Big( \frac{y}{x} \Big)^{\alpha-1}
			&\le C \int_0^1 u^{\alpha-1} (1-u)^{1-\alpha} \, du,
		\end{align}
		where $C$ denotes an unspecified constant (that depends on $\alpha, \beta$ and $\varepsilon$) whose value may change
		from line to line.
		
		Thus,
		\begin{align}
			\sum_{y=1}^{x-1} q(x,y) (y^{\beta}-x^{\beta})
			& = -\beta x^{\beta-1} \sum_{y=1}^{x-1} q(x,y) (x-y) - \beta(1-\beta) \sum_{y=1}^{x-1} q(x,y) \widetilde{R}(x,y)
			\notag \\
			& \ge 
			-\frac{\beta}{(\alpha -1 ) \Gamma(\alpha +1)} x^\beta \frac{\Gamma(x+\alpha)}{\Gamma(x+1)}
			- c_{\alpha,\beta} x^\beta,
		\end{align}
		for a suitable choice of $c_{\alpha,\beta}$ ($\ge \beta/(\alpha-1)$).
	\end{proof}
	
	A direct consequence of \eqref{Generator continuous} and Lemma \ref{lem:Coalescence Generator} is that for functions $g: \R_+ \to \R, g(x)= a x^\beta $ with $0<\beta<2- \alpha$ and $a=a(N) \in \R$  there are constants $c_1,c_2 \geq 0$ such that
	\begin{equation}
		\begin{split}
			a s_N & \left( 1 - \frac{x}{N}\right) x \big((x+1)^\beta-x^\beta\big)- a c_N  c_\alpha \beta \cdot x^\beta\frac{\Gamma(x+\alpha)}{\Gamma(x+1)} - c_1 a c_N x^{\beta }
			\\
			\leq &L^{(N)} g(x)\le a  s_N \left( 1 - \frac{x}{N}\right) x\big((x+1)^\beta-x^\beta\big)- a c_N  c_\alpha \beta \cdot x^\beta\frac{\Gamma(x+\alpha)}{\Gamma(x+1)} + c_2 a c_N x^{\beta },
		\end{split}
		\label{gen beta}
	\end{equation}
	for all $x\in \blue{[N]}$, where
	\begin{equation}
		\label{c_alpha}
		c_\alpha=\frac 1 {(\alpha-1) \Gamma (\alpha+1)}.
	\end{equation}
	
	\subsection{Preparations for and proof of Proposition \ref{proposition}} \label{Section: Proof Prop 1}

	According to Lemma \ref{Lem exact form of the generator} the generator $L^{(N)}$ applied to linear functions $g:\mathbb R_+ \to \mathbb R$, $g(x)=ax$, $a>0$ is of the form
	\begin{align}
		%L^{(N)}g(x)= a s_N \big(1- \frac{x}{N}\big) x -  a c_N x %c_\alpha x^{\alpha-1}
		%+ a c_N x \, \zeta(x)
		L^{(N)}g(x)= a s_N \big(1- \frac{x}{N}\big) x -  a c_N x c_\alpha \frac{\Gamma(x+\alpha)}{\Gamma(x+1)} + a c_N x/(\alpha-1)
	\end{align}
	\blue{with $c_\alpha$ from \eqref{c_alpha}.}
	
	An application of Gautschi's inequality yields 
	\begin{align}
		%x \frac{\Gamma(x+\alpha)}{\Gamma(x+1)} \geq x (x+1)^{\alpha-1} = x^{\alpha} + \widetilde{R}^{(N)}(x), \\
		\blue{x \frac{\Gamma(x+\alpha)}{\Gamma(x+1)} > x(x+\alpha-1)^{\alpha-1} > x^\alpha }.
	\end{align}
	%\red{This would then also simplify the proof from now on?}
	%such that $|\widetilde{R}^{(N)}(x)| \leq C x$, for some constant $C< \infty$ not depending on $x$.
	Hence, we arrive at %(with a modified constant $C$)
	\begin{align}
		L^{(N)}g(x)\leq  a s_N \big(1- \frac{x}{N}\big) x -  a c_N c_\alpha x^{\alpha} + a c_N x/(\alpha-1) .\label{eq:prerequisite}
		%+ a c_N \widetilde{R}^{(N)}(x). 
	\end{align}
	
	We denote by $d_N$ the solution of the equation 
	\begin{align}\label{firstfactor1}
		s_N \big(1- \frac{x}{N}\big) -  c_N c_\alpha x^{\alpha-1}=0.
	\end{align}
	Then
	\begin{align}\label{asym_dN}
		d_N \sim \left( \frac{s_N}{c_\alpha c_N}\right)^{\frac{1}{\alpha-1}},
	\end{align}
	since the left hand side of \eqref{firstfactor1} is $o(1)$ for $x=\left( \frac{s_N}{c_\alpha c_N}\right)^{\frac{1}{\alpha-1}}$.	
	
	Next we show first the upper   bound in Lemma \ref{lm:upper} and then the lower bound in Lemma \ref{lem:lower bound gen}. %Technical results that we need in the proofs are proven in the next subsection.\\

	\begin{lem}\label{lm:upper}
		Let $g^{(N)}(x)=ax$ with
		\[
		a = a(N) = \frac{1}{(\alpha-1)s_N + c_N}
		% \frac{1 + \frac{c_N}{\alpha-1}}{(\alpha-1)s_N + \frac{c_N}{\alpha-1}}
		\]
		(note that $a(N) \sim \frac{1}{(\alpha-1) s_N} $
		as $N\to\infty$ because $c_N/s_N \to 0$).
		%$a= a(N) = \frac 1 {(\alpha-1) s_N}.$ 
		Then % $\forall \, x \in [N]$
		\begin{align}\label{ineqaulity upper bound}
			\limsup_{N\to\infty} \max_{x \in [N]} \frac{L^{(N)} g^{(N)}(x) + x}{d_N} \leq 1.
		\end{align}
	\end{lem}
	
	\begin{proof}
			By \eqref{eq:prerequisite} we have for $x \in [N]$
			\begin{align}\label{Lg_upper.0}
				&L^{(N)} g^{(N)}(x)+x < a \left( s_N + \frac{c_N}{\alpha-1}\right) x - a c_\alpha c_N x^\alpha + x =: f(x)
			\end{align}
			The first and second derivatives of $f$ are 
			\begin{align*}
				f'(x)&=as_N + a \frac{c_N}{\alpha-1} + 1- a \alpha c_{\alpha } c_N x^{\alpha-1} \\
				f''(x)&= -a \alpha (\alpha-1) c_{\alpha} c_N x^{\alpha-2},
			\end{align*}
			and we note that since $a>0$, $f''(x)<0$ for all $x >0$. The function $f$ is maximal at the point 
			\[ 
			x_0= x_0(N) = \left( \frac{s_N + \frac{c_N}{\alpha-1} + \frac{1}{a}}{\alpha c_\alpha c_N}\right)^{\frac{1}{\alpha-1}}
			\]
			and therefore its maximal value is 
			\begin{align}
				f(x_0) &= a x_0 \left( s_N + \frac{c_N}{\alpha-1} - c_\alpha c_N x_0^{\alpha-1} \right) + x_0 \notag \\
				& = a x_0 \left( s_N + \frac{c_N}{\alpha-1} - \frac{s_N}{\alpha} - \frac{c_N}{\alpha(\alpha-1)} - \frac{1}{\alpha a}\right) + x_0 \notag \\
				& = \frac{a x_0}{\alpha} \left( s_N(\alpha-1) + c_N - \frac{1}{a} \right) + x_0 = x_0
				\label{extrema f.0}
			\end{align}
			by the choice of $a=a(N)$. Due to 
			\eqref{choice c_N}, \eqref{conditions_s_b_alpha} 
			with $0<b<\alpha-1$ and the choice of $a=a(N)$ we have $x_0 \sim \left( \frac{s_N}{c_\alpha c_N}\right)^{\frac{1}{\alpha-1}}$ for $N\to\infty$ and the claim follows.
	\end{proof}
	
	\begin{lem} \label{lem:lower bound gen}
		Let $0<  \beta_1 <\beta_2 <\alpha-1$, $N \in \N$ and define $g^{(N)}: \mathbb \N \to \mathbb R$, $g^{(N)}(x)= a_1 x^{\beta_1} + a_2 x^{\beta_2}$ with 
		\begin{align}
			\label{lem:lower bound gen a defs}
			a_1  := a_1(N):= \frac{d_N }{(2^{\beta_1}-1) s_N}, \quad \text{ and } \quad a_2:= a_2(N):=- \frac{\beta_1 a_1 d_N^{\beta_1}}{\beta_2 d_N^{\beta_2}}= -\frac{\beta_1 d_N^{1 + \beta_1 -\beta_2}}{\beta_2 (2^{\beta_1 }-1) s_N}.
		\end{align}
		For this choice %of $g$
		we have
		\begin{align}
			\liminf_{N\to\infty} \min_{x \in \blue{[N]}} \frac{L^{(N)} g^{(N)}(x) + x }{d_N}\geq 1 \label{eq:lem lower boud L}.
		\end{align}
	\end{lem}
	\begin{proof}
		Using Lemma~\ref{lem:Coalescence Generator} \blue{(cf.\ also \eqref{gen beta}) and Gautschi's inequality,} we see that for all $N \in \N$ we have
		\begin{align}
			\label{eq:generator beta 1}
			L^{(N)} g^{(N)}(x)
			& =  a_1 s_N x\left(1- \frac{x}{N}\right)\left(\beta_1 x^{\beta_1 -1} + R^{(N)}_{1,1}(x) \right)  + a_2 s_N x\left(1- \frac{x}{N} \right)\left(\beta_2 x^{\beta_2 -1} + R^{(N)}_{2,1}(x) \right) \notag \\
			& \quad  - a_1 c_N \left( \beta_1  c_\alpha x^{\alpha + \beta_1 -1} + R^{(N)}_{1,2}(x) \right) - a_2 c_N \left( \beta_2  c_\alpha x^{\alpha + \beta_2 -1} + R^{(N)}_{2,2}(x) \right),
		\end{align}
		where the remainder terms satisfy for $x=1,2,\dots,N$ 
		\begin{align}
			|R^{(N)}_{1,1}(x)| \le C_1 x^{\beta_1-2}, \quad |R^{(N)}_{2,1}(x)| \le C_1 x^{\beta_2-2}, \quad
			|R^{(N)}_{1,2}(x)| \le C_1 x^{\beta_1}, \quad |R^{(N)}_{2,2}(x)| \le C_1 x^{\beta_2}
		\end{align}
		with some constant $C_1=C_1(\alpha,b,\beta_1,\beta_2)$ that does not depend on $N$.
		We can re-write \eqref{eq:generator beta 1} as
		\begin{align}
			L^{(N)} g^{(N)}(x)
			& = \left(s_N (1- \frac{x}{N}) - c_\alpha c_N x^{\alpha-1}\right) \left(\beta_1 a_1 x^{\beta_1} + \beta_2 a_2 x^{\beta_2} \right)  \label{eq:generator beta 1a}
			\\ & \qquad +
			% O( a_1 s_N x^{\beta_1-1}) + O(a_2 s_N x^{\beta_2-1}) + O(a_1 c_N x^{\beta_1}) + O(a_2 c_N x^{\beta_2})
			R^{(N)}_1(x) + R^{(N)}_2(x) + R^{(N)}_3(x) + R^{(N)}_4(x)
			\label{eq:generator beta 1a errors}. 
		\end{align}
		with remainder terms satisfying
		\begin{equation}
			\label{eq:generator beta 1a errbds}
			\begin{split}
				|R^{(N)}_1(x)| \le C_2 a_1 s_N x^{\beta_1-1}, & \quad |R^{(N)}_2(x)| \le C_2 a_2 s_N x^{\beta_2-1}, \\
				|R^{(N)}_3(x)| \le C_2 a_1 c_N x^{\beta_1}, & \quad |R^{(N)}_4(x)| \le C_2 a_2 c_N x^{\beta_2}
			\end{split}
		\end{equation}
		uniformly in $x \in [N]$ with some constant $C_2=C_2(\alpha,b,\beta_1,\beta_2)$ that does not depend on $N$.
		\smallskip
		
		\blue{For ease of reference in the following calculations, let
			us recall here from \eqref{conditions_s_b_alpha}, \eqref{choice c_N}, \eqref{asym_dN} that as $N\to\infty$,} 
		\begin{equation}
			\blue{ s_N \sim N^{-b}, \quad c_N \sim (\alpha-1)\Gamma(\alpha) N^{1-\alpha}, \quad
				d_N \sim \left( \frac{s_N}{c_\alpha c_N}\right)^{\frac{1}{\alpha-1}}
				\sim \alpha^{1/(\alpha-1)} N^{1-b/(\alpha-1)} .}
		\end{equation}
		By definition of $d_N$ (recall \eqref{firstfactor1}), the first
		factor in the product on the right-hand side of \eqref{eq:generator beta 1a} is $0$ for $x=d_N$
		and by the choice of $a_1$ and $a_2$, the second factor there is
		also $0$ for $x=d_N$.  Furthermore, the function
		$x \mapsto s_N (1- \frac{x}{N}) - c_\alpha c_N x^{\alpha-1}$ is
		strictly decreasing on $(0,\infty)$; the function
		$x \mapsto h(x) := \beta_1 a_1 x^{\beta_1} + \beta_2 a_2
		x^{\beta_2}$ is strictly decreasing for
		$x > (\beta_1/\beta_2)^{1/(\beta_2-\beta_1)} d_N$ (observe
		$h'(x) = \beta_1^2 a_1 x^{\beta_1-1} + \beta_2^2 a_2 x^{\beta_2-1}
		= a_1(\beta_1^2x^{\beta_1} - \beta_2^2(-a_2/a_1)x^{\beta_2})/x$,
		recall $0 < \beta_1 < \beta_2 < \alpha-1 <1$ and
		$a_2=-(\beta_1/\beta_2) a_1 d_N^{-(\beta_2-\beta_1)}$).
		\medskip
		
		Pick $\varepsilon > 0$ so small that $1-\varepsilon \ge (\beta_1/\beta_2)^{1/(\beta_2-\beta_1)}$.
		\smallskip
		
		\noindent
		\textbf{Case (i)}: Consider $x \ge (1-\varepsilon) d_N$. From \eqref{eq:generator beta 1a errbds}
		and the definitions of $a_1, a_2$ in \eqref{lem:lower bound gen a defs}
		one checks that
		\begin{align}
			\lim_{N\to\infty}
			\max_{x \in [N] \cap [(1-\varepsilon) d_N,\infty)} \frac{|R^{(N)}_1(x)|+|R^{(N)}_2(x)|+|R^{(N)}_3(x)|+|R^{(N)}_4(x)|}{x}
			= 0,
		\end{align}
		by observing that for \blue{ $x \ge (1-\varepsilon) d_N$:}
		\begin{itemize}
			\item $a_1 s_N x^{\beta_1-1} \le C (d_N/x) x^{\beta_1} = o(x)$
			\item $|a_2| s_N x^{\beta_2-1} \le C (d_N^{1-(\beta_2-\beta_1)}/x) x^{\beta_2} = o(x)$
			\item $a_1 c_N x^{\beta_1} \le C d_N^{2-\alpha} x^{\beta_1} \le C x^{2-\alpha+\beta_1} = o(x)$
			\item $|a_2| c_N x^{\beta_2} \le C d_N^{2-\alpha-(\beta_2-\beta_1)} x^{\beta_2} \le C x^{2-\alpha+\beta_1} = o(x)$,	 
		\end{itemize}
		where $C$ denotes a unspecified constant that might change from line to line. Thus we see from \eqref{eq:generator beta 1a}--\eqref{eq:generator
			beta 1a errors} and the discussion above that
		\begin{align}
			\label{eq:lower bound gen case 1}
			\min_{x \in [N], \, x \ge (1-\varepsilon) d_N} \frac{L^{(N)}g^{(N)}(x) + x}{d_N} \ge 1-2\varepsilon,
		\end{align}
		for all large enough $N$.
		\smallskip
		
		\noindent
		\textbf{Case (ii)}: Consider $\varepsilon d_N \le x < (1-\varepsilon) d_N$.
		For such $x$, the remainders in \eqref{eq:generator beta 1a errors}, making use of \eqref{eq:generator beta 1a errbds} and recalling the definition of $a_1,a_2$, satisfy
			\begin{align}
				\sum_{i=1}^4 |R_i^{(N)}(x)| &\leq C \left( d_N x^{\beta_1-1} + d_N^{1+\beta_1-\beta_2} x^{\beta_2-1} + d_N \frac{c_N}{s_N} x^{\beta_1} + \frac{c_N}{s_N} d_N^{1+\beta_1-\beta_2} x^{\beta_2} \right) \notag \\
				&\leq C \left(  d_N x^{\beta_1-1} + d_N^{1+\beta_1-\beta_2} x^{\beta_2-1} + d_N^{2-\alpha} x^{\beta_1}+ d_N^{2-\alpha+\beta_1-\beta_2} x^{\beta_2} \right)
				& %\le C_3 \left( d_N (x^{\beta_1 -1} + d_N^{1-\alpha} x^{\beta_1}) \right) \le C_3 d_N^{2-\alpha +\beta_1} = o(d_N) 
				\label{eq:case.ii.term0}
			\end{align}
			for some constant $C=C(\alpha,b,\beta_1,\beta_2,\eps)$ which might change from line to line, where we used the fact that $c_N/s_N = O(d_N^{1-\alpha})$ by \eqref{asym_dN}. Since $(1-\eps)d_N \geq x \geq \eps d_N$ and $\beta_1 < \alpha-1$ we arrive at 
			\begin{align}\label{eq:EstimateSumRemainderTerms}
				\sum_{i=1}^4 |R_i^{(N)}(x)| &\leq C \left( 2 d_N^{\beta_1} + 2  d_N^{2-\alpha + \beta_1 }\right) \le C d_N^{2-\alpha +\beta_1} < C d_N .
			\end{align}

		%, using $\beta_1 < \alpha -1$ and the asymptotics of $d_N$ in \eqref{asym_dN}.
		For the main term, we analyse the two factors of the product in \eqref{eq:generator beta 1a} separately. Writing $[\varepsilon d_N, (1-\varepsilon)d_N] \ni x = \xi d_N$ with 
		$\xi \in (\eps , (1-\eps) )$, recalling that $d_N$ is the solution to \eqref{firstfactor1}, we have for the first \blue{factor}
		\begin{align}
			s_N \left(1- \frac{\xi d_N}{N}\right) - c_\alpha c_N \left(\xi d_N\right)^{\alpha-1} & = s_N (1-\frac{d_N}{N})- c_{\alpha}c_N d_N^{\alpha-1} \notag \\ & \quad + s_N \frac{1-\xi}{N} d_N + c_{\alpha} c_N \left(d_N^{\alpha-1}-(\xi d_N)^{\alpha -1}\right) \notag \\
			&\geq c_{\alpha}(1-\xi^{\alpha-1})c_Nd_N^{\alpha-1}. \label{eq:term1}
		\end{align}
		On the other hand for the second factor in \eqref{eq:generator beta 1a} we get
		\begin{align}
			\beta_1 a_1 (\xi d_N)^{\beta_1} + \beta_2 a_2 (\xi d_N)^{\beta_2}&= \frac{\beta_1}{2^{\beta_1}-1} \left( \frac{d_N}{s_N} (\xi d_N)^{\beta_1} - \frac{d_N^{1+\beta_1- \beta_2}}{s_N} (\xi d_N)^{\beta_2}\right) \notag \\
			&= \frac{d_N^{1+\beta_1}}{s_N } \frac{\beta_1}{2^{\beta_1}-1} (\xi^{\beta_1}-\xi^{\beta_2}). \label{eq:term2}
		\end{align}
		Putting \eqref{eq:EstimateSumRemainderTerms}, %\eqref{eq:case.ii.term0},
		\eqref{eq:term1} and \eqref{eq:term2} together and noting that $\xi^{\beta_1}-\xi^{\beta_2}>0$ (since $0< \xi < 1$ and $0<\beta_1<\beta_2$) we obtain with $C(\xi) := (c_\alpha \beta_1)(2^{\beta_1}-1)^{-1} (1-\xi^{\alpha-1}) (\xi^{\beta_1}-\xi^{\beta_2})$ (note $\inf_{\varepsilon \le \xi \le 1-\varepsilon} C(\xi)>0$) that 
		\begin{align}
			\min_{\varepsilon d_N \le x \le (1-\varepsilon) d_N} L^{(N)} g^{(N)}(x) & \geq \inf_{\varepsilon \le \xi \le 1-\varepsilon} C(\xi) d_N^{1 + \beta_1} \frac{c_N}{s_N} d_N^{\alpha-1} - C_3 d_N 
			\geq C_4 d_N^{1+\beta_1}
			\label{eq:lower bound gen case 2}
		\end{align}
		with some constant $C_4>0$ for $N$ large enough due to \eqref{asym_dN}.
		\smallskip
		
		\noindent
		\textbf{Case (iii)}: Consider $K \le x < \varepsilon d_N$ for a
		(large) constant $K$ that will be suitably tuned. For such $x$'s,
		when $K$ is large enough and $\varepsilon$ small, one checks that
		the term
		\begin{equation}
			\label{eq:lower bound gen case 2 dom}
			a_1 s_N \beta_1 x^{\beta_1} (1-x/N) = \frac{\beta_1}{2^{\beta_1}-1} d_N x^{\beta_1} (1-x/N)
		\end{equation}
		in \eqref{eq:generator beta 1} dominates all the others (in the
		sense that the sum of the absolute values of all other terms is
		smaller than \blue{$\delta$} times that term for sufficiently large $N$
		for some $\delta=\delta(\beta_1,\beta_2,\varepsilon)>0$, \blue{which can be made arbitrarily small by choosing $\eps$ small}).  Indeed, \blue{since here $x \le d_N=o(N)$ by \eqref{asym_dN}},
		we have
		\begin{align}
			\frac{|a_2| s_N \beta_2 x^{\beta_2} (1-x/N)}{a_1 s_N \beta_1 x^{\beta_1} (1-x/N)}
			& \le \frac{\beta_2 |a_2|}{\beta_1 a_1} x^{\beta_2-\beta_1}
			= d_N^{-(\beta_2-\beta_1)} x^{\beta_2-\beta_1} \le \varepsilon^{\beta_2-\beta_1}, \\
			\frac{a_1 c_N x^{\beta_1+\alpha-1}}{a_1 s_N \beta_1 x^{\beta_1} (1-x/N)}
			& \le C \frac{c_N}{s_N} x^{\alpha-1} = C \frac{x^{\alpha-1}}{d_N^{\alpha-1}} \le C \varepsilon^{\alpha-1}, \\
			\frac{|a_2| c_N x^{\beta_2+\alpha-1}}{a_1 s_N \beta_1 x^{\beta_1} (1-x/N)}
			& \le C d_N^{\beta_1-\beta_2} \frac{c_N}{s_N} x^{\alpha-1+\beta_2-\beta_1}
			= C \frac{x^{\alpha-1+\beta_2-\beta_1}}{d_N^{\alpha-1+\beta_2-\beta_1}}
			\le C \varepsilon^{\alpha-1+\beta_2-\beta_1},
		\end{align}
		with some constants $C=C(\alpha,\beta_1,\beta_2)$. By \eqref{eq:generator beta 1a errbds}  together with \eqref{eq:generator beta 1} $|R_i^{(N)}|$ is smaller than $1/x$ for $i=1,2$ and smaller than $1/x^{\alpha-1}$ for $i=3,4$ times the respective terms, hence we can ignore them.
		
		Furthermore, the quantity in \eqref{eq:lower
			bound gen case 2 dom} is (strictly) larger than $d_N$ as soon as
		$\beta_1 K^{\beta_1} (1-\varepsilon d_N/N)/(2^{\beta_1}-1) > 1$.  Thus
		\begin{align}
			\label{eq:lower bound gen case 3}
			\min_{x \in [N], \, K \le x <  \varepsilon d_N} \frac{L^{(N)}g^{(N)}(x) + x}{d_N} \ge 1
		\end{align}
		for all $N$ large enough.
		\smallskip
		
		\noindent
		\textbf{Case (iv)}: Consider $x \le K$. Here, we work with the general formula \eqref{Generator continuous}
		for the generator. The choice of $a_1,a_2$ implies that there exists a \blue{function} $R^{(N)}_5:\R_+ \to \R$ such that  
		\begin{align}
			L^{(N)} g_N(x)+ x= a_1 s_N \blue{x} ((x+1)^{\beta_1} - x^{\beta_1}) \big(1+R^{(N)}_5(x)\big) +x,
			\quad x=1,2,\dots,K
		\end{align}
		and $\max_{1 \le x \le K} |R^{(N)}_5(x)| \to 0$ as $N\to\infty$. \blue{To argue that $\max_{1 \le x \le K} |R^{(N)}_5(x)| \to 0$ we used that $c_N/s_N \to 0$ and $|a_2|/a_1\to 0$.}
		In particular, due to our choice of $a_1$ we have
		$L^{(N)} g_N(1)\blue{=} d_N (1+R^{(N)}_5(1))$. 
		\blue{ Since $x \mapsto a_1 s_N x ((x+1)^{\beta_1} - x^{\beta_1})$ is increasing for all $x \geq 0$, see \eqref{eq:increasing function} at the end of this proof, we have }
		
		\begin{align}
			\label{eq:lower bound gen case 4}
			\liminf_{N\to\infty} \min_{1 \le x \le K} \frac{L^{(N)}g^{(N)}(x) + x}{d_N} \ge 1.
		\end{align}
		\smallskip
		
		Combining \eqref{eq:lower bound gen case 1}, \eqref{eq:lower bound gen case 2}, \eqref{eq:lower bound gen case 3}
		and \eqref{eq:lower bound gen case 4} we see that the $\liminf$ in \eqref{eq:lem lower boud L}
		is at least $1-2\varepsilon$, then take $\varepsilon \downarrow 0$ to conclude.
		
		\blue{Lastly, we verify that $f(x)=x ((x+1)^{\beta_1} - x^{\beta_1})$ is increasing for $x \geq 0$, by observing
			\begin{align}
				f'(x)&=   (x+1)^{\beta_1} -
				x^{\beta_1} + x \beta_1 ((x+1)^{\beta_1-1} - x^{\beta_1-1})  \notag \\
				&= (x+1)^{\beta_1} \left( 1- (1+\beta_1)\Big( \frac{x}{x+1} \Big)^{\beta_1} + \frac{\beta_1 x}{x+1}  \right) \notag \\
				&\geq (x+1)^{\beta_1} \left( 1- (1+\beta_1) \Big(1- \frac{\beta_1}{x+1} \Big) + \frac{\beta_1 x}{x+1}   \right) \geq 0, \label{eq:increasing function}
			\end{align}
			where we used a generalised Bernoulli inequality (namely, $(1+y)^{\beta_1} \le 1+\beta_1 y$ for $y>-1$ and $0<\beta_1 \le 1$) in the last line.
		}
	\end{proof}

	\begin{proof}[Proof of Proposition~\ref{proposition}]
		The proof is a straightforward consequence of Lemmas~\ref{lm:upper} and \ref{lem:lower bound gen}, noting that $d_N$ defined there fulfils
		\begin{align}
			\frac{d_N}{\alpha^{\frac{1}{\blue{\alpha-1}}} N^{1- \frac b{\alpha-1}}} \to 1, \quad \text{as } N \to \infty. 
		\end{align}
	\end{proof}
	
	\appendix
	\section{Technical results}
	
	In this section we  collect (the proofs) of a couple of results, which we use in the main text. These results are presumably well known, but we recall them and their proofs for completeness and  (partially) due to a lack of a point of reference.

	\begin{proof}[Proof of Lemma \ref{lem:jump rates block counting}] 
		We have
		\begin{align}
			q(x,y)
			& = \binom{x}{x-y+1} \int_0^1 u^{x-y+1} (1-u)^{x-(x-y+1)} \, \frac{\Lambda(du)}{u^2} \notag \\
			%& = \frac{x!}{(x-y+1)! (y-1)!} 
			%\int_0^1 u^{x-y-1} (1-u)^{y-1} \frac{u^{1-\alpha}(1-u)^{\alpha-1} }{\Gamma(2-\alpha)\Gamma(\alpha)} \, du \notag \\
			& = \frac{x!}{(x-y+1)! (y-1)!} \frac{1}{\Gamma(2-\alpha)\Gamma(\alpha)}
			\int_0^1 u^{x-y-\alpha} (1-u)^{y+\alpha-2} \, du \notag \\
			& = \frac{1}{\Gamma(2-\alpha)\Gamma(\alpha)} \frac{\Gamma(x+1)}{\Gamma(x-y+2) \Gamma(y)}
			\frac{\Gamma(x-y-\alpha+1) \Gamma(y+\alpha-1)}{\Gamma(x)} \notag \\
			& = \frac{x}{\Gamma(2-\alpha)\Gamma(\alpha)}
			\frac{\Gamma(y+\alpha-1)}{\Gamma(y)} \frac{\Gamma(x-y-\alpha+1)}{\Gamma(x-y+2)},
			\label{q(x,y).detail}
		\end{align}
		recall $n! = \Gamma(n+1)$ and see also e.g. \cite[p.~2087]{K12}.
	\end{proof}
	
	%The following should be a known identity (for hypergeometric sums?)
	
	\begin{lem} \label{lem:Identities gamma function final 1}
		Let $a, b > -1$, $n \in \N$. We have
		\begin{equation}
			\label{eq:Gamma.frac.sum}
			\sum_{j=1}^n \frac{\Gamma(j+a)}{\Gamma(j+b)}
			= 
			\begin{cases} \displaystyle 
				\frac{1}{a-b+1} \Big( \frac{\Gamma(n+a+1)}{\Gamma(n+b)} - \mathbbm{1}_{b \neq 0}\frac{\Gamma(a+1)}{\Gamma(b)} \Big) & \text{if } b \neq a+1, \\[2ex]
				\displaystyle 
				\sum_{j=1}^n \frac{1}{a+j} & \text{if } b = a+1.
			\end{cases}
		\end{equation}
		In particular for $\alpha \in (1,2)$, 
		\begin{align}
			\sum_{j=1}^{x-1} (x-j) \frac{\Gamma(j+\alpha-1)}{\Gamma(j+1)} =\frac{ x \Gamma(x+\alpha) - x\Gamma(\alpha+1) \Gamma(x+1)}{\alpha (\alpha -1 ) \Gamma(x+1)}.	\label{eq:expectation bock counting jumps}
		\end{align}
	\end{lem}
	\begin{proof}
		The case $b = a+1$ follows from the functional equation of the Gamma function, in particular
		$\Gamma(a+j+1)=(a+j)\Gamma(a+j)$.
		In order to verify the case $b \neq a+1$ note that then
		\begin{align*}
			\frac{1}{a-b+1} \Big( \frac{\Gamma(j+1+a)}{\Gamma(j+b)} - \frac{\Gamma(j+a)}{\Gamma(j-1+b)} \Big)
			&  = \frac{1}{a-b+1} \Big( \frac{(j+a) \Gamma(j+a)}{\Gamma(j+b)}
			- \frac{(b+j-1)\Gamma(j+a)}{(b+j-1)\Gamma(j-1+b)}\Big) \\
			& = \frac{1}{a-b+1} \frac{\Gamma(j+a)}{\Gamma(j+b)}\big( (a+j) - (b+j-1)\big) = \frac{\Gamma(j+a)}{\Gamma(j+b)},
		\end{align*}
		for $j \in \N$ if $b\neq0$ and for $j\in \{2,3,\dots\}$ if $b=0$.
		Summing this over $j=1,2,\dots,n$ gives 
		\eqref{eq:Gamma.frac.sum} in the case $b\neq 0$. In case $b=0$, summing over $j=2,3,\dots,n$
		yields
		\begin{align*}
			\sum_{j=1}^n \frac{\Gamma(j+a)}{\Gamma(j)}
			& = \Gamma(a+1) + \sum_{j=2}^n \frac{\Gamma(j+a)}{\Gamma(j)} \\
			& = \Gamma(a+1) + \frac{1}{a+1} \Big( \frac{\Gamma(n+1+a)}{\Gamma(n)} - \frac{\Gamma(2+a)}{\Gamma(1)}\Big)
			= \frac{\Gamma(n+1+a)}{(a+1)\Gamma(n)},
		\end{align*}
		completing the proof of \eqref{eq:Gamma.frac.sum}. 
		Finally, note that by applying \eqref{eq:Gamma.frac.sum} two times
		\begin{align}
			\sum_{j=1}^{x-1} (x-j) \frac{\Gamma(j+\alpha-1)}{\Gamma(j+1)} &= x \sum_{j=1}^{x-1} \frac{\Gamma(j+\alpha-1)}{\Gamma(j+1)} - \sum_{j=1}^{x-1} \frac{\Gamma(j+\alpha-1)}{\Gamma(j)} \\
			&= \frac{x}{\alpha -1} \left( \frac{\Gamma(x+\alpha-1)}{\Gamma(x)} -\Gamma(\alpha) \right) -  \frac{1}{\alpha} \frac{\Gamma(x+\alpha-1)}{\Gamma(x-1)} \\
			&= \frac{ x \Gamma(x+\alpha) - x\Gamma(\alpha+1) \Gamma(x+1)}{\alpha (\alpha -1 ) \Gamma(x+1)}.	
		\end{align}
	\end{proof}
	\begin{lem}\label{Lem:identity Gamma two terms}
		For $x \in \N$ and $\alpha \in (1,2)$ the following identity holds
		\begin{align*}
			\qquad	 \sum_{y=1}^{x-1} \frac{\Gamma(y+\alpha-1)}{\Gamma(y)} \frac{\Gamma(x-y-\alpha+1)}{\Gamma(x-y+2)}= \frac{ \Gamma(2-\alpha)}{\alpha} \frac{( x-1)\Gamma(x+\alpha-1) }{\Gamma(x+1)}.
		\end{align*}
		
	\end{lem}
	
	\begin{proof}
		Let
		\[a_y:= \frac{\Gamma(y+\alpha-1)}{\Gamma(y)} \frac{\Gamma(x-y-\alpha+1)}{\Gamma(x-y+2)}.\]
		We show that $\sum_{y=1}^{x+1} a_y=0$. This implies \blue{the claim} since
		\begin{align*}
			a_x+a_{x+1}&= \frac{\Gamma(x+\alpha-1)}{\Gamma(x)} \frac{\Gamma(1-\alpha)}{\Gamma(2)} + \frac{\Gamma(x+\alpha)}{\Gamma(x+1)} \frac{\Gamma(-\alpha)}{\Gamma(1)} \\
			&= \frac{-\alpha x \Gamma(x+\alpha-1) \Gamma(1-\alpha)+\Gamma(x+\alpha)\Gamma(1-\alpha)}{-\alpha \Gamma(x+1)} 
			\\
			&= \frac{\Gamma(x+\alpha-1) \Gamma(1-\alpha)( x+\alpha-1-\alpha x )}{-\alpha \Gamma(x+1)} 
			\\
			&= - \frac{ \Gamma(2-\alpha)}{\alpha} \frac{( x-1)\Gamma(x+\alpha-1) }{\Gamma(x+1)}.
		\end{align*}	
		Let $x^{\overline n}:=x(x+1)\cdots (x+n-1)$ and $x^{\underline n}:=x(x-1)\cdots (x-n+1)$ denote the rising and respectively falling factorials. Then 
		\begin{align*}
			& \Gamma(x-y-\alpha) =\frac {\Gamma(x-\alpha)}{(x-1-\alpha)^{\underline y}}= \frac {\Gamma(x-\alpha)}{(-1)^{y}(\alpha+1-x)^{\overline y}},\\
			&\Gamma(x-y+1) =\frac {\Gamma(x+1)}{x^{\underline y}}= \frac {\Gamma(x+1)}{(-1)^{y}(-x)^{\overline y}},
			\\
			& \Gamma(y+\alpha) =\alpha^{\overline y} \Gamma(\alpha)
		\end{align*}	
		and by \cite[15.1.20]{Abramowitz1972}
		\begin{align*}
			\sum_{y=1}^{x+1} a_y = \sum_{y=0}^x \frac{\Gamma(y+\alpha)}{\Gamma(y+1)} \frac{\Gamma(x-y-\alpha)}{\Gamma(x-y+1)}
			&= \frac{\Gamma(\alpha)\Gamma(x-\alpha)}{\Gamma(x+1)} \sum_{y=0}^x \frac 1 {y!} \frac{\alpha^{\overline y}(-x)^{\overline y}}{(-x+1+\alpha)^{\overline y}}
			\\
			& =\frac{\Gamma(\alpha)\Gamma(x-\alpha)}{\Gamma(x+1)} {}_2F_1 (\alpha, -x,-x+1+\alpha;1)
			\\
			&=0,
		\end{align*}
		where ${}_2F_1$ denotes the hypergeometric function.
	\end{proof}
	
	%\begin{lem}\label{lem:Properties Gamma function}
	%The function
	%\begin{equation}\label{eq:Gammafract.mon}
	%\R_+ \ni n \mapsto  \frac{\Gamma(n+b)}{\Gamma(n)} \in \R_+
	%\end{equation}
	%is (strictly) increasing for all $b>0$.
	%{\color{blue} check if this is still needed}
	%\begin{equation}
	%\	\label{eq:Gammafract.bd}
	%\	\text{$\Gamma(n+b)/\Gamma(n) \le n^b$ for $b>0$.}
	%\\end{equation}
	
	%\end{lem}
	
	%\begin{proof}
	%This follows directly by Gautschi's inequality.
	%\  We do not need	{\tt \eqref{eq:Gammafract.bd}}??	{\tt \eqref{eq:Gammafract.bd} follows from log-convexity of the $\Gamma$-function 			(the latter is part of the Bohr-Mollerup theorem); 			should we consult Emil Artin, The gamma function, Holt, Rinehart \& Winston (1964)? }
	%\end{proof}

	\subsection{Proof of formulas from Subsection~\ref{sec:neutral}}
	\label{sect:proofs.rem:neutral}
	
	For a rough idea why $N^{\alpha-1}$ is the correct order of magnitude
	of \eqref{eq:neutral1}, which also highlights the role of the
	singularity of $p^{-2} \Lambda(dp)$ at $0$, we can argue as follows:
	For a given individual among the $N$ many, the rate at which it
	participates in some non-trivial event (i.e., its own $p$-coin toss is
	successful and there is at least one other individual whose $p$-coin
	toss is successful) is given by, ignoring constants,
	\begin{align*}
		& \int_0^1 p \big( 1 - (1-p)^{N-1}\big) \frac{p^{1-\alpha}(1-p)^{\alpha-1}}{p^2} \,dp\\
		&= \int_0^{1/N} p \big( 1 - (1-p)^{N-1}\big) \frac{p^{1-\alpha}(1-p)^{\alpha-1}}{p^2} dp + \int_{1/N}^1 p \big( 1 - (1-p)^{N-1}\big) \frac{p^{1-\alpha}(1-p)^{\alpha-1}}{p^2} dp \\
		& \approx \int_0^{1/N} p(N-1)p \frac{p^{1-\alpha}(1-p)^{\alpha-1}}{p^2} \,dp + 
		\int_{1/N}^1 p \frac{p^{1-\alpha}(1-p)^{\alpha-1}}{p^2} \,dp \\
		& \approx N \int_0^{1/N} p^{1-\alpha} \, dp + \int_{1/N}^1 p^{-\alpha} \, dp \approx N^{\alpha-1}.
	\end{align*}
	
	\paragraph{Proof of \eqref{eq:neutral1}, \eqref{eq:neutral.k} and \eqref{eq:neutral.0}.}
	We observe that
	\begin{align}
		& \int_{(0,1]} p \big( 1 - (1-p)^{N-1}\big) \frac{1}{p^2} \Lambda(dp)
		= \sum_{k=2}^N \int_{(0,1]} \binom{N-1}{k-1} p^{k} (1-p)^{N-k}\frac{1}{p^2} \Lambda(dp) \notag \\
		& = \sum_{k=2}^N \frac{k}{N} \binom{N}{k} \int_{(0,1]} p^{k} (1-p)^{N-k}\frac{1}{p^2} \Lambda(dp)
		= \sum_{k=2}^N \frac{k}{N} q(N, N-k+1) \notag \\
		& = \frac{1}{N} \bigg( \sum_{k=2}^N (k-1) q(N, N-k+1) \bigg)
		+ \frac{1}{N} \bigg( \sum_{k=2}^N q(N, N-k+1) \bigg) \notag \\
		& = \frac{1}{N} \frac{ N \Gamma(N+\alpha) - N\Gamma(\alpha+1) \Gamma(N+1)}{ (\alpha -1 ) \Gamma(\alpha +1) \Gamma(N+1)} + \frac{1}{\Gamma(\alpha+1)} \frac{(N-1) \Gamma(N+\alpha-1)}{\Gamma(N+1)} \label{eq:neutral1.aux1} \\
		& = \frac{1}{(\alpha-1)\Gamma(\alpha+1)} \frac{\Gamma(N+\alpha)}{\Gamma(N+1)}
		- \frac{1}{(\alpha-1)}
		+ \frac{1}{\Gamma(\alpha+1)} \frac{\Gamma(N+\alpha-1)}{\Gamma(N)}
		- \frac{1}{\Gamma(\alpha+1)} \frac{\Gamma(N+\alpha-1)}{\Gamma(N+1)} \notag \\
		& \sim \frac{\alpha}{(\alpha-1)\Gamma(\alpha+1)} N^{\alpha-1},
		\qquad \text{for } N \to\infty, 
	\end{align}
	where we used in \eqref{eq:neutral1.aux1} the identities \eqref{first moment gen} and \blue{that by Lemma~\ref{lem:jump rates block counting}} 
	\begin{align}
		\frac{1}{N} \sum_{k=2}^N q(N, N-k+1)
		& = \frac{1}{N} \sum_{k=2}^N \frac{N}{\Gamma(2-\alpha)\Gamma(\alpha)}
		\frac{\Gamma(N-k+\alpha)}{\Gamma(N-k+1)} \frac{\Gamma(k-\alpha)}{\Gamma(k+1)} \notag \\
		& = \frac{1}{\Gamma(2-\alpha)\Gamma(\alpha)} \sum_{j=1}^{N-1}
		\frac{\Gamma(j+\alpha-1)}{\Gamma(j)} \frac{\Gamma(N-j-\alpha+1)}{\Gamma(N-j+2)} \notag \\
		& = \frac{1}{\Gamma(2-\alpha)\Gamma(\alpha)}
		\frac{(N-1) \Gamma(2-\alpha) \Gamma(N+\alpha-1)}{\alpha \Gamma(N+1)} \\
		&= \frac{1}{\Gamma(\alpha+1)} \frac{(N-1) \Gamma(N+\alpha-1)}{\Gamma(N+1)};
	\end{align}
	here, the third equality is from Lemma~\ref{Lem:identity Gamma two terms} (see also e.g.\ \cite[Appendix]{HM13}).
	
	For \eqref{eq:neutral.k} note that
	\begin{align}
		& \frac{1}{k} \binom{N-1}{k-1} \int_{(0,1]} p \, p^{k-1} (1-p)^{N-k}\frac{1}{p^2} \Lambda(dp)
		= \frac{1}{N} \binom{N}{k} \int_{(0,1]} p^{k} (1-p)^{N-k}\frac{1}{p^2} \Lambda(dp) \\
		&= \frac{1}{N} q(N,N-k+1) \notag
		= \frac{1}{N} \frac{N}{\Gamma(2-\alpha)\Gamma(\alpha)}
		\frac{\Gamma(N-k+\alpha)}{\Gamma(N-k+1)} \frac{\Gamma(k-\alpha)}{\Gamma(k+1)},
	\end{align}
	where we used \eqref{eq:q(x,y)} from Lemma~\ref{lem:jump rates block counting}
	in the second line. 
	
	For \eqref{eq:neutral.0} write
	\begin{align}
		& \sum_{k=2}^N 
		\frac{k-1}{k} \binom{N-1}{k-1} \int_{(0,1]} p \, p^{k-1} (1-p)^{N-k}\frac{1}{p^2} \Lambda(dp)
		= \sum_{k=2}^N \frac{k-1}{N} \binom{N}{k} \int_{(0,1]} p^{k} (1-p)^{N-k}\frac{1}{p^2} \Lambda(dp)
		\notag \\
		& = \frac{1}{N} \sum_{k=2}^N (k-1) q(N,N-k+1)
		= \frac{1}{N} \frac{ N \Gamma(N+\alpha) - N\Gamma(\alpha+1) \Gamma(N+1)}{ (\alpha -1 ) \Gamma(\alpha +1) \Gamma(N+1)}
		\notag \\
		& = \frac{1}{(\alpha -1 ) \Gamma(\alpha +1)} \frac{\Gamma(N+\alpha)}{\Gamma(N+1)}
		- \frac{1}{\alpha -1}
		\sim \frac{1}{(\alpha -1 ) \Gamma(\alpha +1)} N^{\alpha-1},
		\qquad \text{for } N \to\infty,
	\end{align}
	where we used \eqref{first moment gen} in the third equality.
	\hfill $\qed$
	
	\section*{Acknowledgements}
We would like to thank two anonymous reviewers for their friendly and helpful comments which corrected inconsistencies and improved the presentation.
		
	\bibliographystyle{alpha}
	\bibliography{bibfile}

\end{document}